\documentclass[11pt]{amsart}
\usepackage{amsfonts, amssymb, amsmath, amsthm, color, float,enumerate}
\usepackage[unicode,psdextra]{hyperref}
\usepackage{bookmark}
\usepackage{url}
\usepackage{pgfplots,tikz}
\usetikzlibrary{arrows}

\theoremstyle{plain}
   \newtheorem{teo}{Theorem}
   \newtheorem{coro}[teo]{Corollary}
   \newtheorem{lema}[teo]{Lemma}

\theoremstyle{definition}
   
\theoremstyle{remark}
 \newtheorem{obs}{Remark}

\numberwithin{equation}{section}
\numberwithin{teo}{section}
\allowdisplaybreaks

\definecolor{zzttqq}{rgb}{0.6,0.2,0.}
\definecolor{qqzzqq}{rgb}{0.,0.6,0.}

\definecolor{aquamarine}{rgb}{0.5, 1.0, 0.83}
\definecolor{blizzardblue}{rgb}{0.67, 0.9, 0.93}
	\definecolor{blush}{rgb}{0.87, 0.36, 0.51}
	\definecolor{celestialblue}{rgb}{0.29, 0.59, 0.82}
	\definecolor{chocolate(web)}{rgb}{0.82, 0.41, 0.12}

\hypersetup{
	colorlinks = true,
	linkcolor = celestialblue,
	anchorcolor = blue,
	citecolor = blush,
	filecolor = blue,
	urlcolor = chocolate(web)
}

\begin{document}
		
	\title[Optimal parameters related with continuity properties of the $I_{\gamma,m}$]{Optimal  parameters related with continuity properties of the multilinear fractional integral operator between Lebesgue and Lipschitz spaces}

	\author[F. Berra]{Fabio Berra}
	\address{CONICET and Departamento de Matem\'{a}tica (FIQ-UNL),  Santa Fe, Argentina.}
	\email{fberra@santafe-conicet.gov.ar}
	
	\author[G. Pradolini]{Gladis Pradolini}
	\address{CONICET and Departamento de Matem\'{a}tica (FIQ-UNL),  Santa Fe, Argentina.}
	\email{gladis.pradolini@gmail.com}
%
	\author[W. Ramos]{Wilfredo Ramos}
	\address{CONICET and Departamento de Matem\'{a}tica (FaCENA-UNNE),  Corrientes, Argentina.}
	\email{marilcarena@gmail.com}
	
	\thanks{The author were supported by CONICET, UNL, ANPCyT and UNNE}
	
	
	
	\begin{abstract}
		We deal with the boundedness of the multilinear fractional integral operator $I_{\gamma,m}$ from a product of weighted Lebesgue spaces into adequate weighted Lipschitz spaces. Our results generalize some previous estimates not only for the linear case but also for the unweighted problem in the multilinear context.  We characterize the classes of weights for which the problem described above holds and show the optimal range of the parameters involved. The optimality is understood in the sense that the parameters defining the corresponding spaces belong to a certain region. We further exhibit examples of weights for the class which cover the mentioned area. 
	\end{abstract}

	\maketitle
	
	\section{Introduction}\label{seccion: introduccion}
	
	Given $0<\gamma<n$, the classical fractional integral operator $I_\gamma$ is defined by
	\[I_\gamma f(x)=\int_{\mathbb{R}^n}\frac{f(y)}{|x-y|^{n-\gamma}}\,dy,\]
	provided the integral is finite. In \cite{Muckenhoupt-Wheeden74} Muckenhoupt and Wheeden proved that, if $1<p<n/\gamma$ and $1/q=1/p-\gamma/n$, this operator maps $L^p(w^p)$ into $L^q(w^q)$ if and only if $w\in A_{p,q}$. Moreover, when $p=n/\gamma$ they showed that $I_\gamma$ maps $L^{n/\gamma}(w^{n/\gamma})$ into a certain weighted version  of the bounded mean oscillation spaces $\mathrm{BMO}$ if and only if $w^{-n/(n-\gamma)}\in A_1$.
	
	Later on, in \cite{Pradolini01} the author proved that for $n/\gamma<p<n/(\gamma-1)^+$ and $\delta/n=\gamma/n-1/p$ the operator $I_\gamma$ maps $L^p(w^p)$ into a weighted version of Lipschitz spaces associated to the parameter $\delta$. A two-weighted problem it was also studied, giving the optimal parameters for which the associated classes of weights are nontrivial. Other results related with the continuity properties of $I_\gamma$ in the range of $p$ given above can be found in  \cite{CPR16} and for different versions of weighted Lipschitz spaces in \cite{HSV}  and \cite{Prado01cal}.     
	
	Given $m\in\mathbb{N}$ and $0<\gamma<mn$ the multilinear version of order $m$ of the operator above, $I_{\gamma,m}$, is defined as follows
	\[I_{\gamma,m} \vec{f}(x)=\int_{(\mathbb{R}^n)^m} \frac{\prod_{i=1}^m f_i(y_i)}{(\sum_{i=1}^m|x-y_i|)^{mn-\gamma}}\,d\vec{y},\]
	where $\vec{f}=(f_1,f_2,\dots, f_m)$ and $\vec{y}=(y_1,y_2,\dots, y_m)$, provided the integral is finite.
	
	The continuity properties of  $I_{\gamma,m}$ were studied for several authors. For example,  it was shown in \cite{Moen09} that
	$I_{\gamma,m}: \prod_{i=1}^m L^{p_i}\hookrightarrow L^q$, where $1/p=\sum_{i=1}^m1/p_i$ and $1/q=1/p-\gamma/n$. The author also considered weighted versions of these estimates, generalizing the result in \cite{Muckenhoupt-Wheeden74}. On the other hand, in \cite{AHIV} the authors proved unweighted estimates of $I_{\gamma,m}$ between $\prod_{i=1}^m L^{p_i}$ and  Lipschitz-$\delta$ spaces, with $0\leq \delta<1$ and $\delta/n=\gamma/n-1/p$. For other type of estimates involving $I_{\gamma,m}$ see also \cite{Pradolini10}.
	
	In this paper we study the boundedness of the operator $I_{\gamma,m}$ between a product of weighted Lebesgue spaces and certain weighted Lipschitz spaces, generalizing the linear case proved in \cite{Pradolini01} and the unweighted problem given in \cite{AHIV}. We do not only consider related weights, which is an adequate extension of the one-weight estimates in the linear case, but also with independent weights exhibiting a generalization of the two-weight problem for $m=1$. We characterize the classes of weights for which the problem described above holds, by also showing the optimal range of the parameters involved. The optimality is understood in the sense that the parameters defining the corresponding spaces belong to a certain region, becoming trivial outside of it.
	 Moreover we exhibit examples of weights covering this area giving, in this way, a complete theory.   
	  As far as we know, the results in this paper are a first approach to this topic in the weighted multilinear context. 
	  
	We shall now introduce the classes of weights and the notation required in order to state our main results. Throughout the paper the multilinear parameter  will be denoted by $m\in \mathbb{N}$. Let $0<\gamma<mn$, $\delta\in \mathbb{R}$ and $\vec{p}=(p_1,p_2,\dots, p_m)$ be an $m$-tuple of exponents  where $1\le p_i \le \infty$ for $1\le i\le m$. We define $p$ such that $1/p=\sum_{i=1}^{m}1/p_i$. 
	
	Given the weights $w$, $v_1,\dots, v_m$, if $\vec{v}=(v_1,v_2,\dots,v_m)$ we say that the pair $(w,\vec{v})$ belongs to the class $\mathbb{H}_m(\vec{p},\gamma,\delta)$ if there exists a positive constant $C$ such that the inequality
	\begin{equation*}
		\frac{\|w\mathcal{X}_B\|_\infty}{|B|^{(\delta-1)/n}}\prod_{i=1}^m\left(\int_{\mathbb{R}^n} \frac{v_i^{-p_i'}(y)}{(|B|^{1/n}+|x_B-y|)^{(n-\gamma_i+1/m)p_i'}}\,dy\right)^{1/p_i'}\leq C
	\end{equation*} 
	holds for every ball $B$, where $x_B$ denotes the center of $B$ and $\sum_{i=1}^m\gamma_i=\gamma$, with $0<\gamma_i<n$ for every $i$. The integral above is understood as usual when $p_i=1$, (see \S~\ref{section: preliminares} for further details).
	
		When $m=1$ the class defined above was first introduced in \cite{Pradolini01} (see also \cite{Muckenhoupt-Wheeden74} for the case $\delta=0$). In that paper the author showed  nontrivial weights when $\gamma-n\leq \delta\leq\min\{1,\gamma-n/p\}$. We shall see that a similar restriction on $\delta$ appears in the multilinear context.
	
	We are in a position to state our first result.
	
	\begin{teo}\label{teo: teo principal}
		Let $0<\gamma<mn$, $\delta\in\mathbb{R}$,  and $\vec{p}$ a vector of exponents that verifies $p>n/\gamma$. Let $(w,\vec{v})$ be a pair  such that $v_i^{-p_i'}\in \mathrm{RH}_{m}$, for $1\leq i\leq m$. Then the following statements are equivalent:
		\begin{enumerate}[\rm(1)]
		\item \label{item: teo principal item 1} The operator $I_{\gamma,m}$ is bounded from  $\prod_{i=1}^m L^{p_i}(v_i^{p_i})$ to $\mathbb{L}_{w}(\delta)$;
		\item \label{item: teo principal item 2} The pair $(w,\vec{v})$ belongs to $\mathbb{H}_m(\vec{p},\gamma,\delta)$. 
		\end{enumerate}
	\end{teo}

In the linear case the reverse Hölder condition on the weight is trivially satisfied, so our theorem is a well extension of the corresponding result in \cite{Pradolini01} for $p>n/\gamma$.

We have already observed that, although there is no restrictions on $\delta$ in the previous theorem, they arise as a consequence of the nature of the corresponding weights. The next result gives the range of parameters involved in the class defined above where the weights are trivial, that is, $v_i=\infty$ a.e. for some $i$ or $w=0$ a.e. 
	
	\begin{teo}\label{teo: no-ejemplos Hbb} Let $0<\gamma<mn$, $\delta\in\mathbb{R}$,  and $\vec{p}$ a vector of exponents. The following statements hold:
		\begin{enumerate}[\rm(a)]
		\item\label{item: teo no-ejemplos Hbb - item a} If $\delta>1$ or $\delta>\gamma-n/p$ then condition $\mathbb{H}_m(\vec{p},\gamma,\delta)$ is satisfied if and only if $v_i=\infty$ a.e. for some $1\le i\le m$.
		\item\label{item: teo no-ejemplos Hbb - item b} The same conclusion holds if $\delta=\gamma-n/p=1$.
		\item \label{item: teo no-ejemplos Hbb - item c}If $\delta<\gamma-mn$, then condition $\mathbb{H}_m(\vec{p},\gamma,\delta)$ is satisfied if and only if $v_i=\infty$ a.e. for some $1\le i\le m$ or $w=0$ a.e.
		\end{enumerate}	
	\end{teo}

We also exhibit non trivial examples of weights showing that the class is non empty (see \S~\ref{seccion: ejemplos}).

As we shall see if $w=\prod_{i=1}^m v_i$ and $\delta<\tau=(\gamma-mn)(1-1/m)+1/m$ the corresponding class, denoted by $\vec{v}\in \mathbb{H}_m({\vec{p},\gamma,\delta})$, is reduced to the 
$A_{\vec{p},\infty}$ condition. This class is  defined by the vectors $\vec{v}=(v_1,\dots,v_m)$, for which the following inequality 
\[\sup_{B\subset \mathbb{R}^n} \left\|\mathcal{X}_B\prod_{i=1}^m v_i\right\|_\infty\prod_{i=1}^m\left(\frac{1}{|B|}\int_B v_i^{-p_i'}\right)^{1/p_i'}<\infty\]
holds and it is the endpoint case of the $A_{\vec{p},q}$ classes defined in \cite{Moen09}. If $p_i=1$ for some $i$, the corresponding factor above must be understood as $\|v_i^{-1}\mathcal{X}_B\|_\infty$. When $m=1$ this inequality is equivalent to require $v^{-p'}\in A_1$, which is the expected condition (see for example \cite{Pradolini01}). 

The following result summarizes the discussion given above by showing that the  parameters $\delta$ and $p$ are restricted to a line.

\begin{teo}\label{teo: caso pesos iguales}
 Let $0<\gamma<mn$, $\delta\in\mathbb{R}$,  and $\vec{p}$ a vector of exponents. If $\vec{v}\in \mathbb{H}_m({\vec{p},\gamma,\delta})$, then $\delta=\gamma-n/p$. 
\end{teo}

The theorem above proves that if $\delta=\gamma-n/p$, then  $\mathbb{H}_m({\vec{p},\gamma,\delta})\subset A_{\vec{p},\infty}$ and both classes coincide for $\delta<\tau$. When $m=1$ the result above was obtained in \cite{PRR21}.

The article is organized as follows. In \S~\ref{section: preliminares} we give some previous notation and properties of the classes of weights. In \S~\ref{section: resultados auxiliares} we study the behaviour of some operators related with $I_{\gamma,m}$ which will be useful in the proof of the main theorem, given in \S~\ref{section: prueba principal}. Finally, in \S~\ref{seccion: ejemplos} we exhibit examples of weights in the optimal range and prove the result dealing with the particular case of related weights.

\section{Preliminaries}\label{section: preliminares}

Throughout the paper $C$ will denote an absolute constant that may change in every occurrence. By $A\lesssim B$ we mean that there exists a positive constant $c$ such that $A\leq c B$.  We say that $A\approx B$ when $A\lesssim B$ and $B\lesssim A$. 

Let $m\in \mathbb{N}$. Given a set $E$, with $E^m$ we shall denote the cartesian product of $E$ $m$ times.

The multilinear fractional integral operator of order $0<\gamma<mn$ is defined by
\[I_{\gamma,m} \vec{f}(x)=\int_{(\mathbb{R}^n)^m} \frac{\prod_{i=1}^m f_i(y_i)}{(\sum_{i=1}^m|x-y_i|)^{mn-\gamma}}\,d\vec{y},\]
where $\vec{f}=(f_1,f_2,\dots, f_m)$ and $\vec{y}=(y_1,y_2,\dots, y_m)$.
It will be useful for us to consider the operator
\begin{equation}\label{eq: operador Jgamma,m}
J_{\gamma,m}\vec{f}(x)=\int_{(\mathbb{R}^n)^m} \left(\frac{1}{(\sum_{i=1}^m|x-y_i|)^{mn-\gamma}}-\frac{1-\mathcal{X}_{B(0,1)^m}(\vec{y})}{(\sum_{i=1}^m|y_i|)^{mn-\gamma}}\right)\prod_{i=1}^m f_i(y_i)\,d\vec{y}.
\end{equation}
which differs from $I_{\gamma,m}$ only by a constant term. This operator has the same Lipschitz norm as $I_{\gamma,m}$, so it will be enough to give the results for $J_{\gamma,m}$.

By a weight we understand any positive and locally integrable function. 

Given $\delta \in \mathbb{R}$ and a weight $w$ we say that a locally integrable function $f\in  \mathbb{L}_{w}(\delta)$ if there exists a positive constant $C$ such that 
\begin{equation}\label{eq: definicion clase Lipschitz norma inf}
\frac{\|w\mathcal{X}_B\|_\infty}{|B|^{1+\delta/n}}\int_B|f(x)-f_B|\,dx\leq C
\end{equation} 
for every ball $B$, where $f_B=|B|^{-1}\int_B f$. The smallest constant $C$ for which the inequality above holds will be denoted by $\|f\|_{\mathbb{L}_w(\delta)}$.

If $\delta=0$ the space $\mathbb{L}_{w}(\delta)$ coincides with a weighted version of BMO spaces introduced in \cite{Muckenhoupt-Wheeden74}. Concerning to the unweighted case, when $0<\delta<1$ these spaces are equivalent to the classical Lipschitz classes $\Lambda(\delta)$ given by the collection of functions $f$ satisfying $|f(x)-f(y)|\le C |x-y|^{\delta}$ and  they are Morrey spaces when $-n<\delta<0$. These classes of functions were also studied in \cite{Pradolini01}.

As we said in the introduction, the classes  $\mathbb{H}_m(\vec{p},\gamma,\delta)$ are given by the pairs $(w,\vec{v})$ for which the inequality
\begin{equation}\label{eq: clase Hbarra(p,gamma,delta) - m}
\sup_{B\subset \mathbb{R}^n} \frac{\|w\mathcal{X}_B\|_\infty}{|B|^{(\delta-1)/n}}\prod_{i=1}^m\left(\int_{\mathbb{R}^n} \frac{v_i^{-p_i'}(y)}{(|B|^{1/n}+|x_B-y|)^{(n-\gamma_i+1/m)p_i'}}\,dy\right)^{1/p_i'}<\infty 
\end{equation}
holds. For those index $i$ such that $p_i=1$ we understand the corresponding factor on the products above as
\begin{equation}\label{eq: factor de H para p_i=1}
\left\|\frac{v_i^{-1}}{(|B|^{1/n}+|x_B-\cdot|)^{(n-\gamma_i+1/m)}}\right\|_\infty.
\end{equation}
Let $\mathcal{I}_2=\{1\le i\le m: p_i=1\}$ and $\mathcal{I}_2=\{1,\dots, m\}\backslash \mathcal{I}_1$. Observe that  $(w,\vec{v})\in \mathbb{H}_m(\vec{p},\gamma,\delta)$ implies that the inequalities
\begin{equation}\label{eq: condicion local}
\frac{\|w\mathcal{X}_B\|_\infty}{|B|^{\delta/n-\gamma/n+1/p}}\prod_{i\in\mathcal{I}_1}\left\|v_i^{-1}\mathcal X_B\right\|_{\infty}\, \prod_{i\in\mathcal{I}_2}\left(\frac{1}{|B|}\int_{B} v_i^{-p_i'}(y)\,dy\right)^{1/p_i'}\leq C
\end{equation}
and
\begin{equation}\label{eq: condicion global}
\frac{\|w\mathcal{X}_B\|_\infty}{|B|^{(\delta-1)/n}}\prod_{i\in\mathcal{I}_1}\left\|\frac{v_i^{-1}\mathcal X_{\mathbb{R}^n\backslash B}}{|x_B-\cdot|^{(n-\gamma_i+1/m)}}\right\|_{\infty}\, \prod_{i\in\mathcal{I}_2}\left(\int_{\mathbb{R}^n\backslash B} \frac{v_i^{-p_i'}(y)}{|x_B-y|^{(n-\gamma_i+1/m)p_i'}}\,dy\right)^{1/p_i'}\leq C
\end{equation} 
holds for every ball $B$. We shall refer to these inequalities as the \textit{local} and the \textit{global} conditions, respectively.

On the other hand, under certain properties on $\vec{v}$, the corresponding local and global conditions imply \eqref{eq: clase Hbarra(p,gamma,delta) - m}. Before state and prove this result, we shall introduce some useful notation.

	Given $m\in\mathbb{N}$ we denote $S_m=\{0,1\}^m$. Given a set $B$ and $\sigma\in S_m$, $\sigma=(\sigma_1,\sigma_2,\dots,\sigma_m)$ we define 
	\[B^{\sigma_i}=\left\{
	\begin{array}{ccl}
	B,&\textrm{ if }&\sigma_i=1\\
	\mathbb{R}^n\backslash B,&\textrm{ if }&\sigma_i=0.
	\end{array}
	\right.\]
	
	With the notation $\mathbf{B}^\sigma$ we will understand the cartesian product $B^{\sigma_1}\times B^{\sigma_2}\times\dots\times B^{\sigma_m}$. In particular, if we set $\mathbf{1}=(1,1,\dots,1)$ and $\mathbf{0}=(0,0,\dots,0)$ then we have
	\[\mathbf{B}^{\mathbf{1}}=B\times B\times\dots\times B=B^m,\quad\textrm{ and }\quad \mathbf{B}^{\mathbf{0}}=(\mathbb{R}^n\backslash B)\times (\mathbb{R}^n\backslash B)\times \dots\times (\mathbb{R}^n\backslash B)=(\mathbb{R}^n\backslash B)^m.\]
We recall that a weight $w$ belongs to the \textit{reverse H\"{o}lder} class $\mathrm{RH}_s$, $1<s<\infty$, if there exists a positive constant $C$ such that the inequality
\[\left(\frac{1}{|B|}\int_B w^s\right)^{1/s}\leq \frac{C}{|B|}\int_B w\]
holds for every ball $B$ in $\mathbb{R}^n$. The smallest constant for which the inequality above holds is denoted by $[w]_{\mathrm{RH}_s}$. It is not difficult to see that $\mathrm{RH}_t\subset \mathrm{RH}_s$ whenever $1<s<t$. We say that $w\in \mathrm{RH}_{\infty}$ if 
\[\sup_B w\le  \frac{C}{|B|}\int_B w,\]
for some positive constant $C$.	It is well known that any radial power function $|\cdot|^\alpha$, with $\alpha>0$ satisfies $\mathrm{RH}_\infty$ condition.

\begin{lema}\label{lema: equivalencia con local y global}
Let $0<\gamma<mn$, $\delta\in\mathbb{R}$, $\vec{p}$ a vector of exponents and $(w,\vec{v})$ a pair of weights such that $v_i^{-1}\in\mathrm{RH}_\infty$ for $i\in\mathcal{I}_1$ and $v_i^{-p_i'}$ is doubling for $i\in\mathcal{I}_2$. Then, condition $\mathbb{H}_m(\vec{p},\gamma,\delta)$ is equivalent to \eqref{eq: condicion global}.

\end{lema}

\begin{proof}
	We have already seen that $\mathbb{H}_m(\vec{p},\gamma,\delta)$ implies  \eqref{eq: condicion global}. Let $\theta_i=n-\gamma_i+1/m$, for every $i$. If $m_2=m-m_1$ where $m_1=\#\mathcal{I}_1$, the cardinal of $\mathcal{I}_1$, after a possible rename of the index $i\in\mathcal{I}_2$ we have  that
	\[\prod_{i\in\mathcal{I}_2}\left(\int_{\mathbb{R}^n} \frac{v_i^{-p_i'}}{(|B|^{1/n}+|x_B-\cdot|)^{\theta_ip_i'}}\right)^{1/p_i'}=\sum_{\sigma\in S_{m_2}}\prod_{i=1}^{m_0}\left( \int_{B^{\sigma_i}}\frac{v_i^{-p_i'}}{(|B|^{1/n}+|x_B-\cdot|)^{\theta_ip_i'}}\right)^{1/p_i'}.\]
	
	Fix $\sigma\in S_{m_2}$. If $\sigma_i=0$, we have that
	\begin{align*}
	\left(\int_{B^{\sigma_i}}\frac{v_i^{-p_i'}}{(|B|^{1/n}+|x_B-\cdot|)^{(n-\gamma_i+1/m)p_i'}}\right)^{1/p_i'}&=\left(\int_{\mathbb{R}^n\backslash B}\frac{v_i^{-p_i'}}{(|B|^{1/n}+|x_B-\cdot|)^{(n-\gamma_i+1/m)p_i'}}\right)^{1/p_i'}\\
	&\leq \left(\int_{\mathbb{R}^n\backslash B}\frac{v_i^{-p_i'}(y)}{|x_B-y|^{(n-\gamma_i+1/m)p_i'}}\,dy\right)^{1/p_i'}.
	\end{align*}
	
	For $\sigma_i=1$, since $v_i^{-p_i'}$ is doubling, we have that 
	\begin{align*}
	\left(\int_{B^{\sigma_i}}\frac{v_i^{-p_i'}(y)}{(|B|^{1/n}+|x_B-y|)^{(n-\gamma_i+1/m)p_i'}}\,dy\right)^{1/p_i'}&=\left(\int_B\frac{v_i^{-p_i'}(y)}{(|B|^{1/n}+|x_B-y|)^{(n-\gamma_i+1/m)p_i'}}\,dy\right)^{1/p_i'}\\
	&\leq \frac{1}{|B|^{1-\gamma_i/n+1/(mn)}}\left(\int_B v_i^{-p_i'}\right)^{1/p_i'}\\
	& \lesssim \frac{1}{|2B|^{1-\gamma_i/n+1/(mn)}}\left(\int_{2B\backslash B} v_i^{-p_i'}\right)^{1/p_i'}\\
	&\leq\left(\int_{2B\backslash B}\frac{v_i^{-p_i'}(y)}{|x_B-y|^{(n-\gamma_i+1/m)p_i'}}\right)^{1/p_i'}\\
	&\leq\left(\int_{\mathbb{R}^n\backslash B}\frac{v_i^{-p_i'}(y)}{|x_B-y|^{(n-\gamma_i+1/m)p_i'}}\right)^{1/p_i'}.
	\end{align*}
	Therefore, for every $\sigma\in S_{m_2}$ we obtain
	\[\prod_{i=1}^{m_2}\left( \int_{B^{\sigma_i}}\frac{v_i^{-p_i'}}{(|B|^{1/n}+|x_B-\cdot|)^{\theta_ip_i'}}\right)^{1/p_i'}\lesssim \prod_{i\in\mathcal{I}_2}\left( \int_{\mathbb{R}^n\backslash B}\frac{v_i^{-p_i'}}{|x_B-\cdot|^{\theta_ip_i'}}\right)^{1/p_i'}.\]
	On the other hand, when $i\in\mathcal{I}_1$ we can follow a similar argument with the integral replaced by $\|\cdot\|_\infty$. Indeed, since $v_i^{-1}\in \mathrm{RH}_\infty$ observe that
	\[\left\|v^{-1}\mathcal{X}_B\right\|_\infty\leq \frac{\left[v^{-1}\right]_{\mathrm{RH}_\infty}}{|B|}\int_B v^{-1}\leq \frac{C}{|B|}\int_{2B\backslash B} v^{-1}\leq C\left\|v^{-1}\mathcal{X}_{2B\backslash B}\right\|_\infty.\]
	This allows us to estimate as follows
	\[\prod_{i\in\mathcal{I}_1} \left\|\frac{v_i^{-1}}{(|B|^{1/n}+|x_B-\cdot|)^{n-\gamma/m+1/m}}\right\|_\infty\lesssim \prod_{i\in\mathcal{I}_1} \left\|\frac{v_i^{-1}\mathcal{X}_{\mathbb{R}^n\backslash B}}{|x_B-\cdot|^{n-\gamma/m+1/m}}\right\|_\infty.\]
	Therefore, by applying condition \eqref{eq: condicion global}  we have that
	\[\frac{\|w\mathcal{X}_B\|_\infty}{|B|^{(\delta-1)/n}}\prod_{i\in\mathcal{I}_1} \left\|\frac{v_i^{-1}}{(|B|^{1/n}+|x_B-\cdot|)^{n-\gamma/m+1/m}}\right\|_\infty\, \prod_{i\in\mathcal{I}_2}\left(\int_{\mathbb{R}^n} \frac{v_i^{-p_i'}}{(|B|^{1/n}+|x_B-\cdot|)^{\theta_ip_i'}}\right)^{1/p_i'}\leq C,\]
	as desired.
\end{proof}

\begin{coro}
	Under the hypotheses of Lemma~\ref{lema: equivalencia con local y global} we have that conditions \eqref{eq: condicion global}  implies \eqref{eq: condicion local}. 
\end{coro}

The following lemma is a local-to-global result for the condition $\mathbb{H}_m(\vec{p},\gamma,\delta)$. It will be useful in order to give examples of weights. We shall assume that $\gamma_i= \gamma/m$, for every $i$.

\begin{lema}\label{lema: local implica global en Hbarra}
Let $0<\gamma<mn$, $\delta<\tau=(\gamma-mn)(1-1/m)+1/m$, $\vec{p}$ a vector of exponents and $(w,\vec{v})$ a pair of weights satisfying condition \eqref{eq: condicion local}. Then $(w,\vec{v})$ satisfies \eqref{eq: condicion global}.
\end{lema}

\begin{proof}
	Let $\theta=n-\gamma/m+1/m$. Fix a ball $B$ and set $B_k=2^k B$, for every $k\in \mathbb{N}$. If $i\in\mathcal{I}_1$ we have that
	\begin{align*}
	\left\|\frac{v_i^{-1}\mathcal{X}_{\mathbb{R}^n\backslash B}}{|x_B-\cdot|^{\theta}}\right\|_\infty &\leq \sum_{k=1}^\infty \left\|\frac{v_i^{-1}\mathcal{X}_{B_{k+1}\backslash B_k}}{|x_B-\cdot|^{\theta}}\right\|_\infty\\
	&\leq C \sum_{k=1}^\infty |B_k|^{-\theta/n}\left\|v_i^{-1}\mathcal{X}_{B_{k+1}}\right\|_\infty.
	\end{align*}
	On the other hand, for $i\in\mathcal{I}_2$
\begin{align*}
\left(\int_{\mathbb{R}^n\backslash B} \frac{v_i^{-p_i'}(y)}{|x_B-y|^{\theta p_i'}}\,dy\right)^{1/p_i'}&\leq \left(\sum_{k=1}^\infty\int_{B_{k+1}\backslash B_k} \frac{v_i^{-p_i'}(y)}{|x_B-y|^{\theta p_i'}}\,dy\right)^{1/p_i'}\\
&\leq C\sum_{k=1}^\infty |B_k|^{-\theta/n}\left(\int_{B_{k+1}}v_i^{-p_i'}\right)^{1/p_i'}.
\end{align*}
	
	If we set $\vec{k}=(k_1,k_2.\dots,k_m)$, the left-hand side of \eqref{eq: condicion global} can be bounded by 
	\[C\sum_{\vec{k}\in\mathbb{N}^m} \prod_{i\in\mathcal{I}_1} |B_{k_i}|^{-\theta/n}\left\|v_i^{-1}\mathcal{X}_{B_{k_i+1}}\right\|_\infty\, \prod_{i\in\mathcal{I}_2} |B_{k_i}|^{-\theta/n}\left(\int_{B_{k_i+1}}v_i^{-p_i'}\right)^{1/p_i'}=C\sum_{\vec{k}\in\mathbb{N}^m}I\left(B,\vec{k}\right).\]
	Observe that $\mathbb{N}^m\subset \bigcup_{i=1}^m K_i,$
	where $K_i=\{\vec{k}=(k_1,k_2,\dots,k_m): k_i\geq k_j \textrm{ for every }j\}$. Let us estimate the sum over $K_1$, being similar for the other sets. Therefore 
		\begin{align*}
	\sum_{\vec{k}\in K_1} I\left(B,\vec{k}\right)&\leq  \sum_{k_1=1}^\infty|B_{k_1}|^{-\tfrac{\theta}{n}}\prod_{i\in\mathcal{I}_1}\left\|v_i^{-1}\mathcal{X}_{B_{k_1+1}}\right\|_\infty\,\prod_{i\in\mathcal{I}_2}\left(\int_{B_{k_1+1}}v_i^{-p_i'}\right)^{1/p_i'} \prod_{i\neq 1}\sum_{k_i=1}^{k_1}|B_{k_i}|^{-\tfrac{\theta}{n}}.
	\end{align*}
	Notice that
	\[\sum_{k_i=1}^{k_1}|B_{k_i}|^{-\theta/n}=|B|^{-\theta/n}\sum_{k_i=1}^{k_1} 2^{-k_i\theta}\lesssim |B_{k_1}|^{-\theta/n}\sum_{k_i=1}^{k_1}2^{(k_1-k_i)\theta}\lesssim |B_{k_1}|^{-\theta/n}2^{k_1\theta}.\]
	
	Thus, from the estimation above  and \eqref{eq: condicion local} we obtain that
	\begin{align*}
	\frac{\|w\mathcal{X}_B\|_\infty}{|B|^{(\delta-1)/n}}\sum_{\vec{k}\in K_1} I\left(B,\vec{k}\right)&\leq \frac{C}{|B|^{(\delta-1)/n}}\sum_{k_1=1}^\infty 2^{(m-1)k_1\theta}|B_{k_1}|^{-m\theta/n}\|w\mathcal{X}_{B_{k_1+1}}\|_\infty\\
	&\qquad\times \prod_{i\in\mathcal{I}_1}\left\|v_i^{-1}\mathcal{X}_{B_{k_1+1}}\right\|_\infty\, \prod_{i\in\mathcal{I}_2}\left(\int_{B_{k_1+1}}v_i^{-p_i'}\right)^{1/p_i'}\\
	&\leq C\sum_{k_1=1}^\infty 2^{(m-1)k_1\theta} 2^{k_1(\delta-1)},
	\end{align*}
and the last sum is finite provided $\delta<\tau$. 
\end{proof}

\section{Technical results}\label{section: resultados auxiliares}

In this section we introduce some operators involved with $I_{\gamma,m}$ and useful properties in order to prove our main results. 

Let $B=B(x_B,R)$ and $\tilde B=2B$. We can formally decompose the operator in \eqref{eq: operador Jgamma,m} as
\[J_{\gamma,m}\vec{f}(x)=a_B + I\vec{f}(x),\] where
\begin{equation}\label{eq: definicion de a_B}
	a_B=\int_{(\mathbb{R}^n)^m} \left(\frac{1-\mathcal{X}_{\tilde B^m}(\vec y)}{(\sum_{i=1}^m|x_B-y_i|)^{mn-\gamma}}-\frac{1-\mathcal{X}_{B(0,1)^m}(\vec{y})}{(\sum_{i=1}^m|y_i|)^{mn-\gamma}}\right)\prod_{i=1}^m f_i(y_i)\,d\vec{y}
\end{equation}
and
\begin{equation}\label{eq: definicion de I}
	I\vec{f}(x)=\int_{(\mathbb{R}^n)^m} \left(\frac{1}{(\sum_{i=1}^m |x-y_i|)^{mn-\gamma}}-\frac{1-\mathcal{X}_{\tilde B^m}(\vec{y})}{(\sum_{i=1}^m|x_B-y_i|)^{mn-\gamma}}\right)\prod_{i=1}^m f_i(y_i)\,d\vec{y}.
\end{equation}

The first step is to show that this operator is well-defined for $\vec{f}$ as in Theorem~\ref{teo: teo principal}. 

\begin{lema}\label{lema: finitud de J_gamma,m}
	Let $0<\gamma<mn$, $\delta\in\mathbb{R}$,  and $\vec{p}$ be a vector of exponents that verifies $p>n/\gamma$. Let $(w,\vec{v})$ be a pair of weights in $\mathbb{H}_m(\vec{p},\gamma,\delta)$ such that $v_i^{-p_{i}'}\in \mathrm{RH}_{m}$, for $ i\in \mathcal I_2$. If $\vec{f}$ satisfies $f_iv_i\in L^{p_i}$ for every $1\leq i\leq m$, then $J_{\gamma,m}\vec{f}$ is finite in almost every $x\in \mathbb{R}^n$.
\end{lema}

\begin{proof}
		We shall estimate $a_B$ and $I\vec{f}$ separately. Fix $B=B(x_B, R)$ and consider $B_0=B(0, R_0)$, where $R_0=2(|x_B|+R)$. Then, if $\vec{y}\not\in B_0^m$ then the expression between brackets behaves as  
		\[\left(\sum_{i=1}^m|x_B-y_i|\right)^{-mn+\gamma-1}.\]
		We can write
		\begin{align*}
		a_B&=\int_{(\mathbb{R}^n)^m} \left(\frac{1-\mathcal{X}_{\tilde B^m}(\vec y)}{(\sum_{i=1}^m |x_B-y_i|)^{mn-\gamma}}-\frac{1-\mathcal{X}_{B(0,1)^m}(\vec{y})}{(\sum_{i=1}^m|y_i|)^{mn-\gamma}}\right)\left(\prod_{i=1}^mf_i(y_i)\right)\,d\vec{y}\\
		&=\int_{B_0^m}+\int_{(\mathbb{R}^n)^m\backslash B_0^m}\\
		&=a_B^1+a_B^2.
		\end{align*}
		
		We split the estimate of $a_B^1$ into four possible cases.
		\begin{enumerate}
			\item $\mathcal{X}_{\tilde B^m}(\vec{y})=\mathcal{X}_{ B(0,1)^m}(\vec{y})=0$. In this case we have that $y_{i_0}\not\in \tilde B$, for at least one $i_0\in \{1,\dots,m\}$. This yields
			\[\sum_{i=1}^m|x_B-y_i|\geq |x_B-y_{i_0}|>C |B|^{1/n},\]
			since $B\subset B_0$. On the other hand,  $\vec{y}\not\in B(0,1)^m$ implies that $|y_{j_0}|\geq 1$ for some $j_0$. Thus,
			\[\left|\frac{1-\mathcal{X}_{\tilde B^m}(\vec y)}{(\sum_{i=1}^m|x_B-y_i|)^{mn-\gamma}}-\frac{1-\mathcal{X}_{B(0,1)^m}(\vec{y})}{(\sum_{i=1}^m|y_i|)^{mn-\gamma}}\right|\leq 1+\frac{C}{|B|^{m-\gamma/n}}.\]
			\item $\mathcal{X}_{\tilde B^m}(\vec{y})=0$, $\mathcal{X}_{ B(0,1)^m}(\vec{y})=1$. In this case we have that
			\begin{align*}
			\left|\frac{1-\mathcal{X}_{\tilde B^m}(\vec y)}{(\sum_{i=1}^m|x_B-y_i|)^{mn-\gamma}}-\frac{1-\mathcal{X}_{B(0,1)^m}(\vec{y})}{(\sum_{i=1}^m|y_i|)^{mn-\gamma}}\right|&=\frac{1}{(\sum_{i=1}^m|x_B-y_i|)^{mn-\gamma}}\\
			&\leq 1+\frac{C}{|B|^{m-\gamma/n}}.
			\end{align*}
			\item $\mathcal{X}_{\tilde B^m}(\vec{y})=1$, $\mathcal{X}_{ B(0,1)^m}(\vec{y})=0$. We obtain that
			\begin{align*}
			\left|\frac{1-\mathcal{X}_{\tilde B^m}(\vec y)}{(\sum_{i=1}^m|x_B-y_i|)^{mn-\gamma}}-\frac{1-\mathcal{X}_{B(0,1)^m}(\vec{y})}{(\sum_{i=1}^m|y_i|)^{mn-\gamma}}\right|&=\frac{1}{(\sum_{i=1}^m|y_i|)^{mn-\gamma}}\\
			&\leq 1+\frac{C}{|B|^{m-\gamma/n}}.
			\end{align*}
			\item $\mathcal{X}_{\tilde B^m}(\vec{y})=\mathcal{X}_{ B(0,1)^m}(\vec{y})=1$. This is the simplest case since the expression is zero and we trivially obtain the desired bound.
		\end{enumerate}
		
		Recall that $\mathcal{I}_1=\{i: p_i=1\}$ and $\mathcal{I}_2=\{i: p_i>1\}$. With the previous estimate we have that
		\begin{align*}
		|a_B^1|&\leq \left(1+\frac{C}{|B|^{m-\gamma/n}}\right)\left(\int_{B_0^m}\prod_{i=1}^mf_i(y_i)\,d\vec{y}\right)\\
		&=\left(1+\frac{C}{|B|^{m-\gamma/n}}\right)\prod_{i=1}^m\left(\int_{B_0}f_i(y_i)\,dy_i\right)\\
		&\leq \left(1+\frac{C}{|B|^{m-\gamma/n}}\right)
		\prod_{i=1}^m\|f_iv_i\|_{p_i}\prod_{i\in \mathcal{I}_1} \left\|v_i^{-1}\mathcal{X}_{B_0}\right\|_\infty\prod_{i\in \mathcal{I}_2}\left(\int_{B_0}v_i^{-p_i'}\right)^{1/p_i'}\\ 
		&<\infty.
		\end{align*}
		
		We now turn our attention to the estimate of $a_B^2$. By noticing that
		\[(B_0^m)^c=\bigcup_{\sigma\in S_m, \sigma\neq \mathbf{1}} \mathbf{B_0}^{\sigma}\]
		we obtain
		\[a_B^2=\sum_{\sigma\in S_m, \sigma\neq \mathbf{1}}\int_{\mathbf{B_0}^\sigma}\frac{\prod_{i=1}^mf_i(y_i)}{(\sum_{i=1}^m|x_B-y_i|)^{mn-\gamma+1}}\,d\vec{y}.\] 
		Let us estimate a term of this sum for a fixed $\sigma\in S_m$, $\sigma\neq \mathbf{1}$. If we set $\theta_i=n-\gamma_i+1/m$, for $1\leq i\leq m$, we have that
		\begin{equation}\label{eq: lema: finitud de Jgamma,m - eq1}
		\int_{\mathbf{B_0}^\sigma}\frac{\prod_{i=1}^mf_i(y_i)}{(\sum_{i=1}^m|x_B-y_i|)^{mn-\gamma+1}}\,d\vec{y}\leq C\left(\prod_{i: \sigma_i=1}\int_{B_0}\frac{|f_i|}{|B_0|^{\theta_i/n}}\right)\left(\prod_{i: \sigma_i=0}\int_{B_0^c}\frac{|f_i(y_i)|}{|x_B-y_i|^{\theta_i}}\,dy_i\right).
		\end{equation}
		We shall estimate each factor on the right hand side separately. Observe that, for every $i\in\mathcal{I}_1$ such that $\sigma_i=1$, we have
		\[\left\|v_i^{-1}\mathcal{X}_{B_0}\right\|_\infty\leq C |B_0|^{\theta_i/n}\left\|\frac{v_i^{-1}}{(|B_0|^{1/n}+|x_{B_0}-\cdot|)^{\theta_i}}\right\|_\infty.\]
		On the other hand, if $\sigma_i=1$ and $i\in\mathcal{I}_2$ we get
		\begin{align*}
		\frac{1}{|B_0|^{\theta_i/n}}\left(\int_{B_0}v_i^{-p_i'}\right)^{1/p_i'}
		&\leq C\left(\int_{B_0}\frac{v_i^{-p_i'}(y_i)}{(|B_0|^{1/n}+|x_{B_0}-y_i|)^{\theta_ip_i'}}\,dy_i\right)^{1/p_i'}\\
		&\leq C\left(\int_{\mathbb{R}^n}\frac{v_i^{-p_i'}(y_i)}{(|B_0|^{1/n}+|x_{B_0}-y_i|)^{\theta_ip_i'}}\,dy_i\right)^{1/p_i'}.
		\end{align*}

		By combining these two estimates we get
		\begin{align*}
		\prod_{i: \sigma_i=1}\int_{B_0}\frac{|f_i(y_i)|}{|B_0|^{\theta_i/n}}\,dy_i&\leq\left(\prod_{i\in\mathcal{I}_1: \sigma_i=1}\|f_iv_i\|_1 \frac{\left\|v_i^{-1}\mathcal{X}_{B_0}\right\|_\infty}{|B_0|^{\theta_i/n}}\right)\left(\prod_{i\in\mathcal{I}_2: \sigma_i=1} \frac{\|f_iv_i\|_{p_i}}{|B_0|^{\theta_i/n}}\left(\int_{B_0}v_i^{-p_i'}\right)^{1/p_i'}\right)\\
		&\leq C\left(\prod_{i: \sigma_i=1}\|f_iv_i\|_{p_i}\right)\prod_{i\in \mathcal{I}_1: \sigma_i=1}\left\|\frac{v_i^{-1}}{(|B_0|^{1/n}+|x_{B_0}-\cdot|)^{\theta_i}}\right\|_\infty\\
		&\quad \times\prod_{i\in\mathcal{I}_2: \sigma_i=1}\left(\int_{\mathbb{R}^n}\frac{v_i^{-p_i'}(y_i)}{(|B_0|^{1/n}+|x_{B_0}-y_i|)^{\theta_ip_i'}}\,dy_i\right)^{1/p_i'}.
		\end{align*}
		We now turn our attention to the second factor on the right hand side of \eqref{eq: lema: finitud de Jgamma,m - eq1}. If $z\not\in B_0$ then
		\[|z|\leq |z-x_B|+|x_B|\leq |z-x_B|+\frac{R_0}{2}<|z-x_B|+\frac{|z|}{2},\]
		which implies that $|z-x_B|>|z|/2$. Therefore, if $i\in\mathcal{I}_1$ and $\sigma_i=0$ we get
		\[\left\|\frac{v_i^{-1}\mathcal{X}_{B_0^c}}{|x_B-\cdot|^{\theta_i}}\right\|_\infty\leq C \left\|\frac{v_i^{-1}\mathcal{X}_{B_0^c}}{(|B_0|^{1/n}+|x_{B_0}-\cdot|)^{\theta_i}}\right\|_\infty\leq C \left\|\frac{v_i^{-1}}{(|B_0|^{1/n}+|x_{B_0}-\cdot|)^{\theta_i}}\right\|_\infty.\]
		Notice also that, if $i\in\mathcal{I}_2$ and $\sigma_i=0$, we have 
		\begin{align*}
		\left(\int_{B_0^c} \frac{v_i^{-p_i'}(y_i)}{|x_B-y_i|^{\theta_ip_i'}}\,dy_i\right)^{1/p_i'}&\leq C \left(\int_{B_0^c} \frac{v_i^{-p_i'}(y_i)}{(|B_0|^{1/n}+|x_{B_0}-y_i|)^{\theta_ip_i'}}\,dy_i\right)^{1/p_i'}\\
		&\leq C\left(\int_{\mathbb{R}^n} \frac{v_i^{-p_i'}(y_i)}{(|B_0|^{1/n}+|x_{B_0}-y_i|)^{\theta_ip_i'}}\,dy_i\right)^{1/p_i'}.
		\end{align*}	
		Thus we can proceed as follows
		\begin{align*}
		\prod_{i: \sigma_i=0}\int_{B_0^c}\frac{|f_i(y_i)|}{|x_B-y_i|^{\theta_i}}\,dy_i&\leq \prod_{i\in \mathcal{I}_1: \sigma_i=0}\|f_iv_i\|_{1}\left\|\frac{v_i^{-1}\mathcal{X}_{B_0^c}}{|x_B-\cdot|^{\theta_i}}\right\|_\infty\\
		&\quad\times \prod_{i\in\mathcal{I}_2: \sigma_i=0}\|f_iv_i\|_{p_i}\left(\int_{B_0^c}\frac{v_i^{-p_i'}}{|x_B-y_i|^{\theta_ip_i'}}\right)^{1/p_i'}\\
		&\leq C\left(\prod_{i: \sigma_i=0}\|f_iv_i\|_{p_i}\right)\prod_{i\in \mathcal{I}_1: \sigma_i=0}\left\|\frac{v_i^{-1}}{(|B_0|^{1/n}+|x_{B_0}-\cdot|)^{\theta_i}}\right\|_\infty\\
		&\quad \times\prod_{i\in\mathcal{I}_2: \sigma_i=0}\left(\int_{\mathbb{R}^n}\frac{v_i^{-p_i'}(y_i)}{(|B_0|^{1/n}+|x_{B_0}-y_i|)^{\theta_ip_i'}}\,dy_i\right)^{1/p_i'}.
		\end{align*}
		By using these estimates in \eqref{eq: lema: finitud de Jgamma,m - eq1} and applying condition \eqref{eq: clase Hbarra(p,gamma,delta) - m} we obtain that 
		\[	\int_{\mathbf{B_0}^\sigma}\frac{\prod_{i=1}^mf_i(y_i)}{(\sum_{i=1}^m|x_B-y_i|)^{mn-\gamma+1}}\,d\vec{y}\leq C \frac{|B_0|^{(\delta-1)/n}}{\|w\mathcal{X}_{B_0}\|_\infty}\prod_{i=1}^m \|f_iv_i\|_{p_i},\]
		and therefore
		\[a_B^2\leq C\frac{|B_0|^{(\delta-1)/n}}{\|w\mathcal{X}_{B_0}\|_\infty}\prod_{i=1}^m \|f_iv_i\|_{p_i}.\]
		We proceed with the estimation of $I\vec{f}$. We write $I\vec{f}(x)=I_1\vec{f}(x)+I_2\vec{f}(x)$, where
		\[I_1\vec{f}(x)=\int_{\tilde B^m}\frac{\prod_{i=1}^mf_i(y_i)}{(\sum_{i=1}^m|x-y_i|)^{mn-\gamma}}\,d\vec{y}\]
		and
		\[I_2\vec{f}(x)=\int_{(\tilde B^m)^c}\left(\prod_{i=1}^mf_i(y_i)\right)\left(\frac{1}{(\sum_{i=1}^m|x-y_i|)^{mn-\gamma}}-\frac{1}{(\sum_{i=1}^m|x_B-y_i|)^{mn-\gamma}}\right)\,d\vec{y}.\]
		Let us first estimate $I_1$. We shall split the set $\mathcal I_2$ into $\mathcal I_2^1$ and $\mathcal I_2^2$ where
		\[\mathcal{I}_2^1=\{i: 1<p_i<\infty\}\quad \textrm{ and } \quad\mathcal{I}_2^2=\{i: p_i=\infty\}.\]
		Let $m_2^j=\#\mathcal{I}_2^j$, for $j=1,2$. Then $m=m_1+m_2=m_1+m_2^1+m_2^2$. Observe that
		\begin{align*}
		|I_1\vec{f}(x)|&\leq \int_{\tilde B^m}\frac{\prod_{i=1}^m|f_i(y_i)|}{(\sum_{i=1}^m|x-y_i|)^{mn-\gamma}}\,d\vec{y}\\
		&\leq \left(\prod_{i\in\mathcal{I}_2^2}\|f_iv_i\|_\infty\right)\int_{\tilde B^m}\frac{\prod_{i\in \mathcal{I}_1\cup\mathcal{I}_2^1}|f_i(y_i)|\prod_{i\in \mathcal{I}_2^2}v_i^{-1}(y_i)}{(\sum_{i=1}^m|x-y_i|)^{mn-\gamma}}\,d\vec{y}\\
		&\leq \left(\prod_{i\in\mathcal{I}_2^2}\|f_iv_i\|_\infty\right)\left(\prod_{i\in \mathcal{I}_1}\|f_i\mathcal{X}_{\tilde B}\|_1\right) \int_{\tilde B^{m_2}}\frac{\prod_{i\in \mathcal{I}_2^1}|f_i(y_i)|\prod_{i\in \mathcal{I}_2^2}v_i^{-1}(y_i)}{(\sum_{i\in \mathcal{I}_2}|x-y_i|)^{mn-\gamma}}\,d\vec{y}\\
		&=\left(\prod_{i\in\mathcal{I}_2^2}\|f_iv_i\|_\infty\right)\left(\prod_{i\in \mathcal{I}_1}\|f_i\mathcal{X}_{\tilde B}\|_1\right)I(x,B).
		\end{align*}
		Since $p>n/\gamma$ we have that
		\[\gamma>n/p=n\sum_{i=1}^m\frac{1}{p_i}=m_1n+\frac{n}{p^*},\]
		where $1/p^*=\sum_{i\in\mathcal{I}_2}1/p_i$. This allows us to split $\gamma=\gamma^1+\gamma^2$, where $\gamma^1>m_1n$ and $\gamma^2>n/p^*$. Therefore
		\[mn-\gamma=m_2n-\gamma^2+m_1n-\gamma^1.\]
		Let us sort the sets $\mathcal{I}_2^1$ and $\mathcal{I}_2^2$ increasingly, so
		\[\mathcal{I}_2^1=\left\{i_1,i_2,\dots,i_{m_2^1}\right\} \quad \textrm{ and } \quad \mathcal{I}_2^2=\left\{i_{m_2^1+1},i_{m_2^1+2},\dots,i_{m_2}\right\}.\]
		We now define $\vec{g}=(g_1,\dots,g_{m_2})$, where
		\[g_j=\left\{\begin{array}{ccl}
		|f_{i_j}|&\textrm{ if }&1\leq j\leq m_2^1;\\
		v_{i_j}^{-1}&\textrm{ if }&m_2^1+1\leq j\leq m_2.
		\end{array}
		\right.
		\]
		Then we can estimate as follows
		\begin{align*}
		I(x,B)&\leq C\int_{\tilde B^{m_2}}\frac{\prod_{i\in \mathcal{I}_2^1}|f_i(y_i)|\prod_{i\in \mathcal{I}_2^2}v_i^{-1}(y_i){(\sum_{i\in \mathcal{I}_2}|x-y_i|)^{{\gamma^1-nm_1}}}}{(\sum_{i\in \mathcal{I}_2}|x-y_i|)^{m_2n-\gamma^2}}\,d\vec{y}\\
		&\leq C|\tilde B|^{\gamma^1/n-m_1}\int_{\tilde B^{m_2}}\frac{\prod_{i\in \mathcal{I}_2^1}|f_i(y_i)|\prod_{i\in \mathcal{I}_2^2}v_i^{-1}(y_i)}{(\sum_{i\in \mathcal{I}_2}|x-y_i|)^{m_2n-\gamma^2}}\,d\vec{y}\\
		&=C|\tilde B|^{\gamma^1/n-m_1}\int_{\tilde B^{m_2}}\frac{\prod_{j=1}^{m_2}g_j(y_{i_j})}{(\sum_{j=1}^{m_2}|x-y_{i_j}|)^{m_2n-\gamma^2}}\,d\vec{y}\\
		&\leq C|\tilde B|^{\gamma^1/n-m_1}I_{\gamma^2,m_2}(\vec{g}\mathcal{X}_{\tilde B^{m_2}})(x).
		\end{align*}
		Next we define the vector of exponents $\vec{r}=(r_1,\dots,r_{m_2})$ in the following way
		\[r_j=\left\{\begin{array}{ccr}
		m_2p_{i_j}/(m_2-1+p_{i_j})&\textrm{ si }&1\leq j\leq m_2^1;\\
		m_2&\textrm{ si }&m_2^1+1\leq j\leq m_2.
		\end{array}
		\right.
		\]
		This definition yields
		\begin{align*}
		\frac{1}{r}&=\sum_{j=1}^{m_2}\frac{1}{r_j}=\sum_{j=1}^{m_2^1}\left(\frac{1}{m_2}+\frac{m_2-1}{m_2p_{i_j}}\right)+\sum_{j=m_2^1+1}^{m_2}\frac{1}{m_2}=\frac{m_2^1}{m_2}+\frac{m_2-1}{m_2{p^*}}+\frac{m_2^2}{m_2}=1+\frac{m_2-1}{m_2p^*}.
		\end{align*}
		
		Observe that $1/r>1/p^*$. We also have $n/p^*<\gamma^2$. Then there exists an auxiliary number $\gamma_0$ such that $n/p^*<\gamma_0<n/r$. Indeed, if $\gamma^2<n/r$ we can directly pick $\gamma_0=\gamma^2$. Otherwise $\gamma_0<\gamma^2$. Let us first assume that $m_2\geq 2$. We set
		\[\frac{1}{q}=\frac{1}{r}-\frac{\gamma_0}{n}.\]
		Then $0<1/q<1$ since
		\[\frac{1}{r}-1=\left(1-\frac{1}{m_2}\right)\frac{1}{p^*}<\frac{1}{p^*}<\frac{\gamma_0}{n}.\]
				By using the fact that $I_{\gamma_0,m_2}: \prod_{j=1}^{m_2} L^{r_j}\to L^q$ (see \cite{Moen09}) we obtain 
		\begin{align*}
		\int_B I(x,B)\,dx&\leq C|\tilde B|^{\gamma^1/n-m_1+(\gamma^2-\gamma_0)/n}\left(\int_B|I_{\gamma_0,m_2}(\vec{g}\mathcal{X}_{\tilde B^{m_2}})(x)|^q\,dx\right)^{1/q}|B|^{1/q'}\\
		&\leq C|\tilde B|^{(\gamma-\gamma_0)/n-m_1+1/q'}\left(\int_{\mathbb{R}^n}|I_{\gamma_0,m_2}(\vec{g}\mathcal{X}_{\tilde B^{m_2}})(x)|^q\,dx\right)^{1/q}\\
		&\leq C|\tilde B|^{(\gamma-\gamma_0)/n-m_1+1/q'}\prod_{j=1}^{m_2}\|g_j\mathcal{X}_{\tilde B}\|_{r_j}.
		\end{align*}
		Observe that $r_j<p_{i_j}$ for every $1\leq j\leq m_2^1$. By applying H\"{o}lder inequality we have
		\begin{align*}
		\prod_{j=1}^{m_2}\|g_j\mathcal{X}_{\tilde B}\|_{r_j}&=\prod_{i\in\mathcal{I}_2^1}\left(\int_{\tilde B}|f_i|^{r_i}v_i^{r_i}v_i^{-r_i}\right)^{1/r_i}\prod_{i\in \mathcal{I}_2^2}\left(\int_{\tilde B}v_i^{-m_2}\right)^{1/m_2}\\
		&\leq \prod_{i\in\mathcal{I}_2^1}\|f_iv_i\|_{p_i}\left(\int_{\tilde B} v_i^{-m_2p_i'}\right)^{1/(m_2p_i')}\prod_{i\in \mathcal{I}_2^2}\left(\int_{\tilde B}v_i^{-m_2}\right)^{1/m_2}\\
		&\leq |\tilde B|^{m_2^1/m_2-1/(m_2p^*)+m_2^2/m_2}\prod_{i\in\mathcal{I}_2^1}\left[v_i^{-p_i'}\right]_{\mathrm{RH}_{m_2}}\|f_iv_i\|_{p_i}\left(\frac{1}{|\tilde B|}\int_{\tilde B} v_i^{-p_i'}\right)^{1/p_i'}\\
		&\quad \times \prod_{i\in \mathcal{I}_2^2}\left[v_i^{-1}\right]_{\mathrm{RH}_{m_2}}\left(\frac{1}{|\tilde B|}\int_{\tilde B}v_i^{-1}\right)\\
		&=C|\tilde B|^{1-1/(m_2p^*)}\prod_{i\in\mathcal{I}_2^1}\|f_iv_i\|_{p_i}\left(\frac{1}{|\tilde B|}\int_{\tilde B} v_i^{-p_i'}\right)^{1/p_i'}\prod_{i\in \mathcal{I}_2^2}\left(\frac{1}{|\tilde B|}\int_{\tilde B}v_i^{-1}\right).
		\end{align*}
		By combining all these estimates with condition \eqref{eq: condicion local}, we finally get that
		\begin{align*}
		\int_B |I_1\vec{f}(x)|\,dx&\leq \left(\prod_{i\in\mathcal{I}_2^2}\|f_iv_i\|_\infty\right)\left(\prod_{i\in \mathcal{I}_1}\|f_i\mathcal{X}_{\tilde B}\|_1\right)\int_B I(x,B)\,dx\\
		&\leq C\left(\prod_{i\in\mathcal{I}_2^2}\|f_iv_i\|_\infty\right)\left(\prod_{i\in \mathcal{I}_1}\|f_i\mathcal{X}_{\tilde B}\|_1\right)|\tilde B|^{(\gamma-\gamma_0)/n-m_1+1/q'+1-1/(m_2p^*)}\\
		&\quad \times \prod_{i\in\mathcal{I}_2^1}\|f_iv_i\|_{p_i}\left(\frac{1}{|\tilde B|}\int_{\tilde B} v_i^{-p_i'}\right)^{1/p_i'}\prod_{i\in \mathcal{I}_2^2}\left(\frac{1}{|\tilde B|}\int_{\tilde B}v_i^{-1}\right)\\
		&\leq C\left(\prod_{i=1}^m \|f_iv_i\|_{p_i}\right)\prod_{i\in \mathcal{I}_2}\left(\frac{1}{|\tilde B|}\int_{\tilde B} v_i^{-p_i'}\right)^{1/p_i'}\prod_{i\in\mathcal{I}_1} \left\|v_i^{-1}\mathcal{X}_{\tilde B}\right\|_\infty \\
		&\quad \times |\tilde B|^{(\gamma-\gamma_0)/n-m_1+1/q'+1-1/(m_2p^*)}\\
		&\leq C \|w\mathcal{X}_{\tilde B}\|_\infty^{-1}\, |\tilde B|^{\delta/n-\gamma/n+1/p+(\gamma-\gamma_0)/n-m_1+1/q'+1-1/(m_2p^*)}\\
		&\leq C \|w\mathcal{X}_{B}\|_\infty^{-1}|B|^{1+\delta/n}.
		\end{align*}
		Therefore, we can obtain the desired bound for $I_1\vec{f}$ provided $m_2\geq 2$. We now consider the case $0\leq m_2<2$. There are only three possible cases:  
		\begin{enumerate}
			\item $m_2=0$. In this case we have $m_2^1=m_2^2=0$ and this implies $\vec{p}=(1,1,\dots,1)$. This situation is not possible, because $p>n/\gamma$.
			\item $m_2^1=0$ and $m_2^2=1$. In this case $1/p=m-1$. The condition $p>n/\gamma$ implies $\gamma>(m-1)n$. Let $i_0$ be the index such that $p_{i_0}=\infty$. By using Fubini's theorem, we can proceed in the following way
			\[\int_B \int_{\tilde B^m}\frac{\prod_{i=1}^m|f_i(y_i)|}{(\sum_{i=1}^m|x-y_i|)^{mn-\gamma}}\,d\vec{y}\,dx=\int_{\tilde B^m}\prod_{i=1}^m|f_i(y_i)|\left(\int_B \left(\sum_{i=1}^m|x-y_i|\right)^{\gamma-mn}\,dx\right)\,d\vec{y}.\]
			Since
			\begin{align*}
			\int_B \left(\sum_{i=1}^m|x-y_i|\right)^{\gamma-mn}\,dx&\leq C\int_0^{4R}\rho^{\gamma-mn}\rho^{n-1}\,d\rho\\
			&\leq C |B|^{\gamma/n-m+1},
			\end{align*}
			by  \eqref{eq: condicion local}, we get
			\begin{align*}
			\int_B |I_1\vec{f}(x)|\,dx&\leq C|B|^{\gamma/n-m+2}\left(\prod_{i=1}^m\|f_iv_i\|_{p_i}\right)\left(\prod_{i\in \mathcal{I}_1}\left\|v_i^{-1}\mathcal{X}_{\tilde B}\right\|_\infty\right) \left(\frac{1}{|\tilde B|}\int_{\tilde B}v_{i_0}^{-1}\right)\\
			&\leq C\left(\prod_{i=1}^m\|f_iv_i\|_{p_i}\right)\frac{|\tilde B|^{\gamma/n-m+2+\delta/n-\gamma/n+1/p}}{\|w\mathcal{X}_{\tilde B}\|_\infty}\\
			&\leq C\left(\prod_{i=1}^m\|f_iv_i\|_{p_i}\right)\frac{|B|^{1+\delta/n}}{\|w\mathcal{X}_{ B}\|_\infty}.
			\end{align*}
			\item $m_2^1=1$ and $m_2^2=0$. If $i_0$ denotes the index for which $1<p_{i_0}<\infty$, the condition $p>n/\gamma$ implies that
			\[\frac{\gamma}{n}>\frac{1}{p}=m-1+\frac{1}{p_{i_0}},\]
			and thus $\gamma>(m-1)n$. We repeat the estimate given in the previous case. Then
			\begin{align*}
			\int_B |I_1\vec{f}(x)|\,dx&\leq C|B|^{\gamma/n-m+1+1/p_{i_0}'}\left(\prod_{i=1}^m\|f_iv_i\|_{p_i}\right)\left(\prod_{i\in \mathcal{I}_1}\left\|v_i^{-1}\mathcal{X}_{\tilde B}\right\|_\infty\right)\left(\frac{1}{|\tilde B|}\int_{\tilde B}v_{i_0}^{-p_{i_0}'}\right)^{1/p_{i_0}'}\\
			&\leq C\left(\prod_{i=1}^m\|f_iv_i\|_{p_i}\right)\frac{|\tilde B|^{\gamma/n-m+1+1/p_{i_0}'+\delta/n-\gamma/n+1/p}}{\|w\mathcal{X}_{\tilde B}\|_\infty}\\
			&\leq C\left(\prod_{i=1}^m\|f_iv_i\|_{p_i}\right)\frac{|B|^{1+\delta/n}}{\|w\mathcal{X}_{ B}\|_\infty}.
			\end{align*}
		\end{enumerate}

	This completes the estimate for $I_1\vec{f}$.
	For $I_2\vec f$, by the mean value theorem, we can write
		\[|I_2\vec{f}(x)|\leq |B|^{1/n}\sum_{\sigma\in S_m,\sigma\neq \mathbf{1}} \int_{\mathbf{\tilde B}^\sigma} \frac{\prod_{i=1}^m|f_i(y_i)|}{(\sum_{i=1}^m|x_B-y_i|)^{mn-\gamma+1}}\,d\vec{y}.\]
		Notice that this expression is similar to $a_B^2$, with $B_0$ replaced with $\tilde B$. Therefore, we can proceed in a similar way to obtain
		\[\int_B |I_2\vec{f}(x)|\,dx\leq C\frac{|B|^{1+\delta/n}}{\|w\mathcal{X}_B\|_\infty}\prod_{i=1}^m\|f_iv_i\|_{p_i}.\]
		This completes the proof of the lemma.
\end{proof}

\begin{obs}
The corresponding estimate obtained for $I\vec{f}$ will be used for the proof of Theorem~\ref{teo: teo principal}. 
\end{obs}

Next we are going to set some geometrical facts that will be useful later. These results were set and proved in \cite{Pradolini01}. For a fixed ball $B=B(x_B,R)$ we define the sets 
\[A=\{x_B+h: h=(h_1,h_2,\dots,h_n): h_i\geq 0 \textrm{ for }1\leq i\leq n\},\]
\[C_1=B\left(x_B-\frac{R}{12\sqrt{n}}u,\frac{R}{12\sqrt{n}}\right)\cap\left\{x_B-\frac{R}{12\sqrt{n}}u+h: h_i\leq 0 \textrm{ for every }i\right\},\]
and
\[C_2=B\left(x_B-\frac{R}{3\sqrt{n}}u,\frac{2R}{3}\right)\cap\left\{x_B-\frac{R}{3\sqrt{n}}u+h: h_i\leq 0 \textrm{ for every }i\right\},\]
where $u=(1,1,\dots,1)$.  The following figure shows a sketch of these sets.
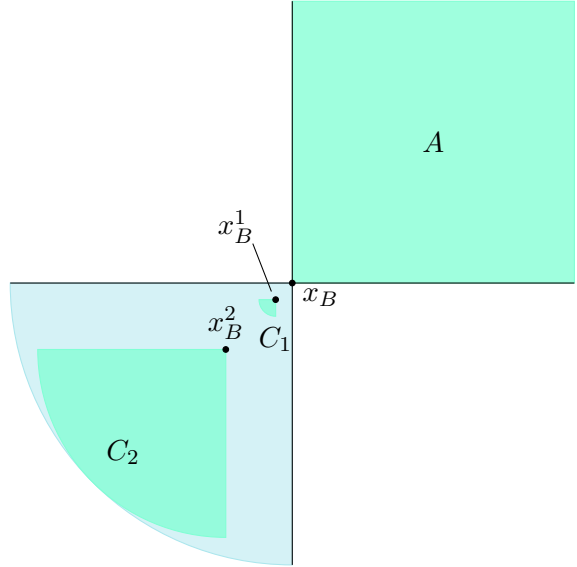
\begin{figure}[h!]
\begin{center}
	\begin{tikzpicture}[scale=0.75]
	\draw[color=aquamarine, fill=aquamarine, fill opacity=0.75] (0,0)--(5,0)--(5,5)--(0,5)--cycle;
	\draw[color=blizzardblue, fill=blizzardblue, fill opacity=0.5] (0,0)--(-5,0) arc(180:270:5)--cycle;
	\node at (2.5,2.5) {$A$};
	\draw(0,-5)--(0,5);
	\draw(-5,0)--(5,0);
	\draw[fill=black] (0,0) circle (0.05cm);
	\draw[color=aquamarine, fill=aquamarine, fill opacity=0.75](-0.2946,-0.2946)--(-0.5892,-0.2946) 
	arc(180:270:0.2946cm)--cycle;
	\draw[fill=black] (-0.2946,-0.2946) circle (0.05cm);
	\draw[color=aquamarine, fill=aquamarine, fill opacity=0.75](-1.1785,-1.1785)--(-4.5118,-1.1785) arc(180:270:3.3333cm)--cycle;
	\draw[fill=black] (-1.1785,-1.1785) circle (0.05cm);
	\node [right] at (0,-0.25) {$x_B$};
	\node [below] at (-0.3,-0.6) {$C_1$};
	\node at (-3,-3) {$C_2$};
	\node[above] at (-1.1785,-1.1785) {$x_B^2$};
	\node at (-1,1) {$x_B^1$};
	\draw(-0.7,0.7)--(-0.37,-0.15);
	\end{tikzpicture}	
\end{center}
\caption{The sets $A$, $C_1$ and $C_2$, where $x_B^1=x_B-R/(12\sqrt{n})u$ and $x_B^2=x_B-R/(4\sqrt{n})u$.}
\label{fig: conjuntos A, C1 y C2}	
\end{figure}

\begin{obs}\label{obs: medida de conjuntos C como B}
It is not difficult to see that $|C_i|\approx |B|$, for $i=1,2$.
\end{obs}
The next lemma deals with the sets defined above and will be useful in the proof of our main result. The proof is similar to the corresponding result given in \cite{Pradolini01} for the case $m=1$ and we omit it. 

%

\begin{lema}\label{lema: diferencia de nucleos positiva}
	There exists a positive constant $C=C(n)$ such that the inequality
	\[\frac{1}{(\sum_{j=1}^m|x-y_j|)^{mn-\gamma}}-\frac{1}{(\sum_{j=1}^m|z-y_j|)^{mn-\gamma}}\geq C\frac{|B|^{1/n}}{(|B|^{1/n}+\sum_{j=1}^m|x_B-y_j|)^{mn-\gamma+1}}\]
	holds for every $x\in C_1$, $z\in C_2$, and $y_j\in A$ for $1\leq j\leq m$.
\end{lema}


\section{Proof of the main results}\label{section: prueba principal}

We devote this section to prove the results contained in Section~\ref{seccion: introduccion}.

\begin{proof}[Proof of Theorem~\ref{teo: teo principal}]
	We shall first prove that $(\ref{item: teo principal item 2})$ implies $(\ref{item: teo principal item 1})$. We shall deal with the operator $J_{\gamma,m}$ since it differs from $I_{\gamma,m}$ by a constant term. We want to prove that for every ball $B$
	\begin{equation}\label{eq: teo principal - eq1}
	\frac{\|w\mathcal{X}_B\|_\infty}{|B|^{1+\delta/n}}\int_B |J_{\gamma,m}\vec{f}(x)-(J_{\gamma,m}\vec{f})_B|\,dx\leq C\prod_{i=1}^m\|f_iv_i\|_{p_i},
	\end{equation}
	with $C$ independent of $B$. Fix a ball $B=B(x_B,R)$ and recall that $J_{\gamma,m}\vec{f}(x)=a_B+I\vec{f}(x)$. In  Lemma~\ref{lema: finitud de J_gamma,m} we proved that
	\begin{equation*}
	\int_B|I\vec{f}(x)|\,dx\leq C\frac{|B|^{1+\delta/n}}{\|w\mathcal{X}_B\|_\infty}\prod_{i=1}^m\|f_iv_i\|_{p_i},
	\end{equation*}
	which implies that
	\begin{equation}\label{eq: teo principal - estimacion de If}
	\int_B|J_{\gamma,m}\vec{f}(x)-a_B|\,dx\leq C\frac{|B|^{1+\delta/n}}{\|w\mathcal{X}_B\|_\infty}\prod_{i=1}^m\|f_iv_i\|_{p_i}.
	\end{equation}
	On the other hand, observe that
	\begin{align*}
	\int_B |J_{\gamma,m}\vec f (x)-(J_{\gamma,m}\vec f)_B|\,dx&\leq \int_B|J_{\gamma,m}\vec{f}(x)-a_B|\,dx+\int_B|(J_{\gamma,m}\vec f)_B-a_B|\,dx\\
	&\leq \int_B|J_{\gamma,m}\vec{f}(x)-a_B|\,dx+\int_B\frac{1}{|B|}\int_B|J_{\gamma,m}\vec f(y)-a_B|\,dy\,dx\\
	&\leq 2\int_B|J_{\gamma,m}\vec{f}(x)-a_B|\,dx.
	\end{align*}
	By combining this estimate with \eqref{eq: teo principal - estimacion de If} we obtain the desired inequality.
	
	We now prove that $(\ref{item: teo principal item 1})$ implies $(\ref{item: teo principal item 2})$. Assume that the component functions $f_i$ of $\vec{f}$ are nonnegative. We have that \eqref{eq: teo principal - eq1} holds for every ball $B=B(x_B,R)$. Also observe that
	\[\frac{1}{|B|}\int_B|g(x)-g_B|\,dx\approx \frac{1}{|B|^2}\int_B\int_B|g(x)-g(z)|\,dx\,dz,\]
	and therefore the left hand side of \eqref{eq: teo principal - eq1} is equivalent to
	\[\frac{\|w\mathcal{X}_B\|_\infty}{|B|^{2+\delta/n}}\int_B\int_B |J_{\gamma,m}\vec{f}(x)-J_{\gamma,m}\vec{f}(z)|\,dx\,dz=:I.\]
	Observe that, when $y_i\in B$ for every $i$ we have  
	\[|B|^{1/n}+|x_B-y_j|\geq \frac{1}{m}\left(|B|^{1/n}+\sum_{i=1}^m|x_B-y_i|\right),\]
	for every $1\leq j\leq m$.
	By combining Lemma~\ref{lema: diferencia de nucleos positiva} and Remark~\ref{obs: medida de conjuntos C como B} with the inequality above we can estimate $I$ as follows
	\begin{align*}
	I&\geq \frac{\|w\mathcal{X}_B\|_\infty}{|B|^{2+\delta/n}}\int_{C_2}\int_{C_1} \int_{A^m} \frac{|B|^{1/n}\prod_{i=1}^m f_i(y_i)}{(|B|^{1/n}+\sum_{i=1}^m|x_B-y_i|)^{mn-\gamma+1}}\,d\vec{y}\,dx\,dz\\
	&\geq C\frac{\|w\mathcal{X}_B\|_\infty}{|B|^{(\delta-1)/n}}\prod_{i=1}^m\left(\int_A \frac{f_i(y_i)}{(|B|^{1/n}+|x_B-y_i|)^{n-\gamma_i+1/m}}\,dy_i\right).
	\end{align*}
	Since the set $A$ is a quadrant from $x_B$, a similar estimation can be obtained for the other quadrants from $x_B$. Thus, we get
	\[I\geq C\frac{\|w\mathcal{X}_B\|_\infty}{|B|^{(\delta-1)/n}}\prod_{i=1}^m\left(\int_{\mathbb{R}^n} \frac{f_i(y)}{(|B|^{1/n}+|x_B-y|)^{n-\gamma_i+1/m}}\,dy\right),\]
	which implies that
	\begin{equation}\label{eq: teo principal - eq2}
	\frac{\|w\mathcal{X}_B\|_\infty}{|B|^{(\delta-1)/n}}\prod_{i=1}^m\left(\int_{\mathbb{R}^n} \frac{f_i(y)}{(|B|^{1/n}+|x_B-y|)^{n-\gamma_i+1/m}}\,dy\right)\leq C\prod_{i=1}^m\|f_iv_i\|_{p_i}.
	\end{equation}
	For every $i\in \mathcal{I}_1$ and $k\in\mathbb{N}$ we define $V_k^i=\{x: v_i^{-1}(x)\leq k\}$ and the functionals
	\[F_i^k(g)=\int_{\mathbb{R}^n}\frac{g(y)v_i^{-1}(y)\mathcal{X}_{V_k^i}(y)}{(|B|^{1/n}+|x_B-y|)^{n-\gamma_i+1/m}}\,dy.\] 
	Therefore $F_i^k$ is a functional in $(L^1)^*=L^{\infty}$. Indeed, if $g\in L^1$
	\[|F_i^k(g)|\leq \|g\|_{L^1} \left\|\frac{v_i^{-1}\mathcal{X}_{V_k^i}}{(|B|^{1/n}+|x_B-\cdot|)^{n-\gamma_i+1/m}}\right\|_\infty<\infty,\]
	and we also get
	\[\frac{|F_i^k(f_iv_i)|}{\|f_iv_i\|_{L^1}}\leq \left\|\frac{v_i^{-1}\mathcal{X}_{V_k^i}}{(|B|^{1/n}+|x_B-\cdot|)^{n-\gamma_i+1/m}}\right\|_\infty,\]
	for every $i\in\mathcal{I}_1$.
	
	If $i\in \mathcal{I}_2$ then we set $A_k=A\cap B(0,k)$ and consider 
	\[f_i^k (y)=\frac{v_i^{-p_i'}(y)}{(|B|^{1/n}+|x_B-y|)^{(n-\gamma_i+1/m)/(p_i-1)}}\mathcal{X}_{A_k}(y)\mathcal{X}_{V_k^i}(y).\]
	
	Let us choose $\vec f=(f_1,\dots,f_m)$, where $f_i$ is such that $f_iv_i\in L^1$ for $p_i=1$ and $f_i=f_i^k$ for $p_i>1$, for $k$ fixed. Therefore, the left hand side of \eqref{eq: teo principal - eq2} can be written as follows 
	\[\frac{\|w\mathcal{X}_B\|_\infty}{|B|^{(\delta-1)/n}}\prod_{i\in \mathcal{I}_1}F_i^k(f_iv_i)\prod_{i\in \mathcal{I}_2}\left(\int_{A_k\cap V_k^i} \frac{v_i^{-p_i'}(y)}{(|B|^{1/n}+|x_B-y|)^{(n-\gamma_i+1/m)p_i'}}\,dy\right)\]
	and it is bounded by
	\[ C\prod_{i\in\mathcal{I}_1}\|f_iv_i\|_{L^1}\prod_{i\in\mathcal{I}_2}\left(\int_{A_k\cap V_k^i} \frac{v_i^{-p_i'}(y)}{(|B|^{1/n}+|x_B-y|)^{(n-\gamma_i+1/m)p_i'}}\,dy\right)^{1/p_i}.\]
	This yields
	\[\frac{\|w\mathcal{X}_B\|_\infty}{|B|^{(\delta-1)/n}}\prod_{i\in\mathcal{I}_1}\frac{|F_i^k(f_iv_i)|}{\|f_iv_i\|_{L^{1}}}\prod_{i\in\mathcal{I}_2}\left(\int_{A_k\cap V_k^i} \frac{v_i^{-p_i'}(y)}{(|B|^{1/n}+|x_B-y|)^{(n-\gamma_i+1/m)p_i'}}\,dy\right)^{1/p_i'}\leq C,\]
	for every nonnegative $f_i$ such that $f_iv_i\in L^1$, $i\in\mathcal{I}_1$ and for every $k\in\mathbb{N}$. By taking the supremum over these $f_i$ we get
	\[\frac{\|w\mathcal{X}_B\|_\infty}{|B|^{(\delta-1)/n}}\prod_{i\in\mathcal{I}_1}\left\|\frac{v_i^{-1}}{(|B|^{1/n}+|x_B-\cdot|)^{n-\gamma_i+1/m}}\right\|_\infty\prod_{i\in\mathcal{I}_2}\left(\int \frac{v_i^{-p_i'}\mathcal{X}_{A_k\cap V_k^i}}{(|B|^{1/n}+|x_B-\cdot|)^{(n-\gamma_i+1/m)p_i'}}\right)^{\tfrac{1}{p_i'}}\leq C.\]
	By taking limit for $k\to\infty$, the left hand side converges to
	\[\frac{\|w\mathcal{X}_B\|_\infty}{|B|^{(\delta-1)/n}}\prod_{i\in\mathcal{I}_1}\left\|\frac{v_i^{-1}}{(|B|^{1/n}+|x_B-\cdot|)^{n-\gamma_i+1/m}}\right\|_\infty\prod_{i\in\mathcal{I}_2}\left(\int_{\mathbb{R}^n} \frac{v_i^{-p_i'}(y)}{(|B|^{1/n}+|x_B-y|)^{(n-\gamma_i+1/m)p_i'}}\,dy\right)^{1/p_i'},\]
	which is precisely the condition $\mathbb{H}_m(\vec{p},\gamma,\delta)$. This completes the proof.\qedhere

\end{proof}

We now proceed to prove Theorem~\ref{teo: no-ejemplos Hbb}.

\begin{proof}[Proof of Theorem~\ref{teo: no-ejemplos Hbb}]
	Let us begin with item~\eqref{item: teo no-ejemplos Hbb - item a}.  We shall first assume that 
	$\delta >1$.  If $(w,\vec{v}) \in~\mathbb{H}_m(\vec{p},\gamma,\delta)$, we choose $B=B(x_B, R)$ where $x_B$ is a Lebesgue point of $w^{-1}$. From \eqref{eq: clase Hbarra(p,gamma,delta) - m} we obtain 
	\begin{align*}
	\prod_{i\in\mathcal{I}_1}\left\|\frac{v_i^{-1}}{(|B|^{1/n}+|x_B-\cdot|)^{n-\gamma_i+1/m}}\right\|_\infty\,\prod_{i\in\mathcal{I}_2}\left(\int_{\mathbb{R}^n} \frac{v_i^{-p_i'}}{(|B|^{1/n}+|x_B-\cdot|)^{(n-\gamma_i+1/m)p_i'}}\right)^{\tfrac{1}{p_i'}}&\lesssim \frac{|B|^{(\delta-1)/n}}{\|w\mathcal{X}_B\|_\infty}\\
	&\lesssim \frac{w^{-1}(B)}{|B|R^{1-\delta}},
	\end{align*}
	for every $R>0$. By letting $R\to 0$ and applying the monotone convergence theorem, we conclude that at least one limit factor in the product should be zero. That is, there exists $1\leq i \leq m$ such that $v_i = \infty$ almost everywhere. 
	
	On the other hand, if $\delta > \gamma - n/p$ and $(w,\vec{v})$ belongs to  $\mathbb{H}_m(\vec{p},\gamma,\delta)$, we pick a ball $B=B(x_B, R)$, where $x_B$ is a Lebesgue point of $w^{-1}$ and every $v_i^{-1}$. Then we have
	\begin{align*}
	\prod_{i=1}^m \frac{1}{|B|}\int_B v_i^{-1}&\leq \prod_{i\in\mathcal{I}_1}\left\|v_i^{-1}\mathcal{X}_B\right\|_\infty\prod_{i\in\mathcal{I}_2}\left( \frac{1}{|B|}\int_B v_i^{-p'_i} \right)^{1/p'_i }\\
	&\leq C \frac{|B|^{\frac{\delta}{n}-\frac{\gamma}{n}+\frac{1}{p}  } }{||w\mathcal{X}_B||_{\infty}}\\
	&\leq C \frac{w^{-1}(B)}{|B|} R^{\delta - \gamma + n/p}
	\end{align*}	
	for every $R>0$. By letting $R\to 0$ we get
	\begin{equation*}
	\prod_{i=1 }^{m} v_{i}^{-1}(x_B)=0, 
	\end{equation*}	
	which yields that $\prod_{i=1 }^{m} v_{i}^{-1}$ is zero almost everywhere. This implies that the set $M=\bigcap_{i=1}^m \{v_i^{-1} >0 \}$ has null measure. Since $v_i(y)>0$ for almost every $y$ and every $i$, there exists $j$ such that $v_j = \infty$ almost everywhere.
	
	We turn now our attention to \eqref{item: teo no-ejemplos Hbb - item b}. Suppose $\delta= \gamma - n/p =1$. We shall prove that if $(w,\vec{v})\in~\mathbb{H}_m(\vec{p}, \gamma, 1)$, there exists $j$ such that $v_j = \infty$ in almost $\mathbb{R}^n$. We define
	\begin{equation*}
	\alpha = \sum_{i=1 }^{m }\frac{1}{p'_i} = m - \frac{1}{p}.
	\end{equation*}
	By applying Hölder inequality we obtain that
	\begin{equation*}
	\left(\int_{\mathbb{R}^n } \frac{(\prod_{i\in\mathcal{I}_2} v_i^{-1 })^{1/\alpha}}{(|B|^{1/n} + |x_B -y|)^{\sum_{i\in\mathcal{I}_2}(n-\gamma_i +1/m)/\alpha }} \right)^{\alpha} 
	\leq C \prod_{i\in\mathcal{I}_2} 
	\left(
	\int_{\mathbb{R}^n } \frac{ v_i^{-p'_i }}{(|B|^{1/n} + |x_B -\cdot|)^{(n-\gamma_i + 1/m)p'_i} } \right)^{\tfrac{1}{p_i'}}
	\end{equation*}
	and since $(w,\vec{v})\in \mathbb{H}_{m}(\vec{p}, \gamma, 1)$ this implies that
	\[\prod_{i\in \mathcal{I}_1}\left\|\frac{v_i^{-1}}{(|B|^{1/n}+|x_B-\cdot|)^{n-\gamma_i+1/m}}\right\|_\infty\left(\int_{\mathbb{R}^n } \frac{(\prod_{i\in\mathcal{I}_2} v_i^{-1 })^{1/\alpha}}{(|B|^{1/n} + |x_B -y|)^{\sum_{i\in\mathcal{I}_2}(n-\gamma_i +1/m)/\alpha }} \right)^{\alpha}\lesssim \frac{w^{-1}(B)}{|B|},\]
	and furthermore
	\[\left(\int_{\mathbb{R}^n } \frac{(\prod_{i=1}^m v_i^{-1 })^{1/\alpha}}{(|B|^{1/n} + |x_B -y|)^{(mn-\gamma +1)/\alpha }}\right) ^{\alpha}\lesssim \frac{w^{-1}(B)}{|B|}\]
	for every ball $B=B(x_B,R)$.
	
	We now use an adaptation of an argument of \cite{Pradolini01}. If the set $E=\{x: ~ \prod_{i=1 }^m v_i^{-1 }(x)>~0\}$ has positive measure we write $E=\bigcup_{k\geq 1 }E_k$, where $E_k= \{x: ~ \prod_{i=1 }^m v_i^{-1 }(x)> 1/k\}$. Then there exists  $k_0$ verifying $|E_{k_0 }| >0$. Let $x_B$ be a Lebesgue point of $w^{-1}$ which also is a density point in $E_{k_0 }$ such that $w^{-1}(x_B)<\infty$. By letting $R\to 0$ and observing that $(mn-\gamma +1)/\alpha = n$ we obtain
	\begin{equation*}
	\left(\int_{\mathbb{R}^n } \frac{(\prod_{i=1 }^m v_i^{-1 })^{1/\alpha}}{ |x_B -y|^{n}} \right)^{\alpha}
	=
	\left(\int_{\mathbb{R}^n } \frac{(\prod_{i=1 }^m v_i^{-1 })^{1/\alpha}}{ |x_B -y|^{(mn-\gamma+1)/\alpha}} \right)^{\alpha} 
	\leq Cw^{-1}(x_B),
	\end{equation*}
	so the left-hand side integral is finite for almost every $x_B\in E_{k_0}$. On the other hand, we have
	\begin{equation}\label{e3}
	\left(\int_{\mathbb{R}^n } \frac{(\prod_{i=1 }^m v_i^{-1 })^{1/\alpha}}{ |x_B -y|^{n}} \right)^{\alpha}
	\geq 
	\left(\int_{B\cap E_{k_0 }} \frac{(\prod_{i=1 }^m v_i^{-1 })^{1/\alpha}}{ |x_B -y|^{n}} \right)^{\alpha}
	\geq \frac{1}{k_0} 
	\left(\int_{B\cap E_{k_0 }} \frac{dy}{ |x_B -y|^{n}} \right)^{\alpha}.
	\end{equation}
	We now define $\varepsilon (R):= (|B|-|B\cap E_{k_0 }|)^{1/n }=|B\backslash E_{k_0 }|^{1/n }$. Since $x_B$ is a density point of $E_{k_0 }$ we have that 
	\begin{equation*}
	\frac{\varepsilon (R)}{R} = c \left(\frac{|B| - |B\backslash E_{k_0 }|}{|B|} \right)^{1/n }\to  0
	\end{equation*}	
	when $R$ approaches to zero. If we take $C= \{y:~ \varepsilon(R)\leq |x_B -y|\leq R \}$, then
	\begin{equation*}
	|C|=c(R^n - \varepsilon^{n}(R))= c (|B|-(|B|-|B\cap E_{k_0 }|))
	=|B\cap E_{k_0}|.
	\end{equation*}
	We also have  
	\begin{equation*}
	|C|=|B\cap E_{k_0 }|=|(B\cap E_{k_0 })\cap C| + |(B\cap E_{k_0 })-C|
	\end{equation*}
	and therefore 
	\begin{equation*}
	|C\backslash(B\cap E_{k_0 }\cap C)|= |C|-|B\cap E_{k_0 }\cap C|
	= |B\cap E_{k_0 }\backslash C|.
	\end{equation*}
	Since 
	\begin{equation*}
	\sup_{C\backslash(B\cap E_{k_0 }\cap C)}|x_0 - y|^{-n}
	\leq 	\inf_{B\cap E_{k_0 }\backslash C}|x_0 - y|^{-n}
	\end{equation*}
	we return to \eqref{e3} and write 
	\begin{align*}
	\left(\int_{\mathbb{R}^n } \frac{(\prod_{i=1 }^m v_i^{-1 })^{1/\alpha}}{ |x_B -y|^{n}} \right)^{\alpha}
	&\geq 
	\frac{1}{k_0} \left(\int_{B\cap E_{k_0} \cap C } \frac{dy}{|x_B -y|^{n}} 
	+
	\int_{(B\cap E_{k_0}) \backslash C } \frac{dy}{|x_B -y|^{n}}
	\right)^{\alpha}\\
	&\geq 
	\frac{1}{k_0} \left(\int_{B\cap E_{k_0} \cap C } \frac{dy}{|x_B -y|^{n}} 
	+
	\int_{C\backslash(B\cap E_{k_0} \cap C) } \frac{dy}{|x_B -y|^{n}}
	\right)^{\alpha}\\
	&=\frac{1}{k_0}\left(\int_{C } \frac{dy}{|x_B -y|^{n}} \right)^{\alpha} = \frac{1}{k_0} \left(\int_{\varepsilon(R) }^R \frac{dr}{r}\right)^{\alpha }\\
	&= \frac{1}{k_0}\ln\left(\frac{R}{\varepsilon(R)}\right)^{\alpha},
	\end{align*}
	which approaches to $\infty$ when $R\to 0$, a contradiction. This yields 
	$|E|=0$, that is, $\prod_{i=1 }^m v_i^{-1} =0$ almost everywhere, from where we can deduce that there exists  $j$ satisfying $v_j = \infty$ almost everywhere.
	
	We finish with the proof of item~\eqref{item: teo no-ejemplos Hbb - item c}. If $\delta<\gamma-mn$,  given a ball $B=B(x_B, R)$ and $B_0\subset B$ the condition \eqref{eq: condicion local} implies that
	\begin{equation*}
	\|w\mathcal{X}_{B_0}\|_\infty \prod_{i=1}^{m}\|v_i^{-1 }\mathcal{X}_{B_0}\|_{p'_i}\leq \|w\mathcal{X}_{B}\|_\infty \prod_{i=1 }^{m}\|v_i^{-1 }\mathcal{X}_{B}\|_{p'_i}\leq CR^{\delta -\gamma+mn}.
	\end{equation*}
	Observe that the right-hand side of the inequality above tends to zero when $R$ tends to $\infty$. This implies that either $\|w\mathcal{X}_{B_0}\|_\infty =0$ or $\|v_i^{-1 }\mathcal{X}_{B_0}\|_{p'_i}=0$, for some  $i$. By the arbitrariness of $B_0$ we obtain either $w=0$ or $v_i =\infty$ for some $i$, respectively.
\end{proof}

\section{The class \texorpdfstring{$\mathbb{H}_m(\vec{p},\gamma,\delta)$}{$Hm(p,\gamma,\delta)$}}\label{seccion: ejemplos}
 We begin this section by exhibiting nontrivial pairs of weights satisfying condition $\mathbb{H}_m(\vec{p},\gamma,\delta)$. Concretely, we shall prove the following theorem.
\begin{teo}\label{teo: ejemplos para Hbb}
	Given $0<\gamma<mn$ there exist pairs of weights $(w,\vec{v})$ satisfying \eqref{eq: clase Hbarra(p,gamma,delta) - m} for every $\vec{p}$ and $\delta$ such that $\gamma-mn\leq \delta\leq \min\{1,\gamma-n/p\}$, excluding the case $\delta=1$ when $\gamma-n/p=1$.
\end{teo} 

The following figure shows the area in which we can find nontrivial weights, depending on the value of $\gamma$.

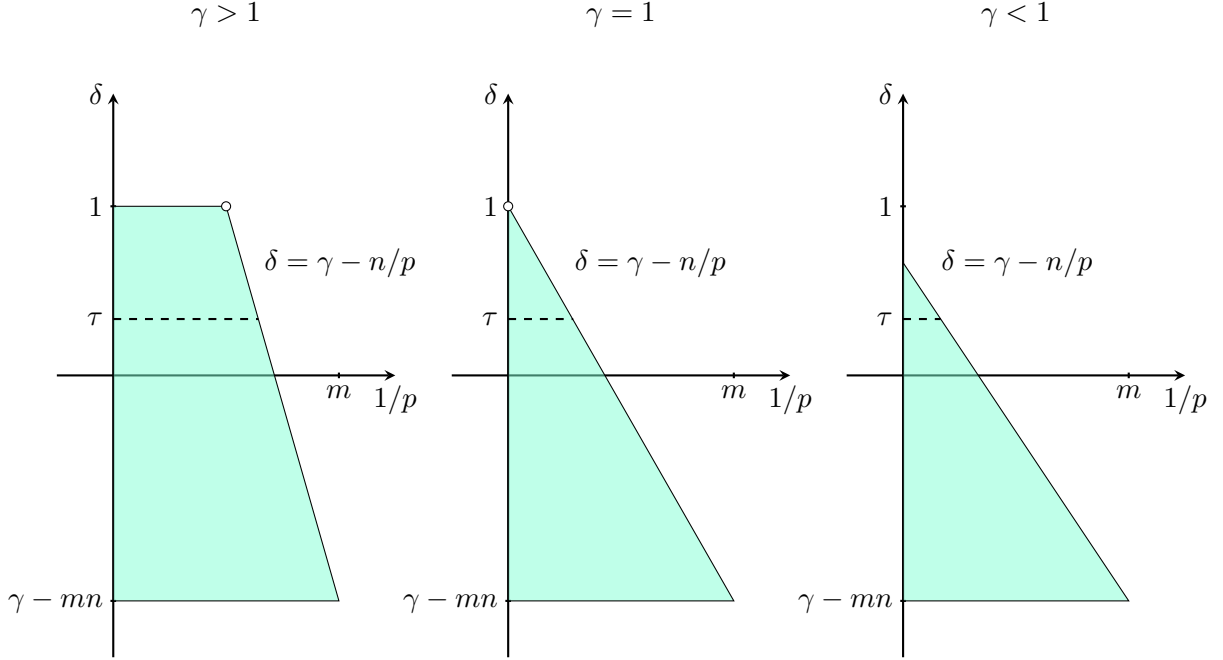
\begin{center}
	\begin{tikzpicture}[scale=0.75]
	\node[above] at (-4,6) {$\gamma>1$};
	\draw [-stealth, thick] (-6,-5)--(-6,5);
	\draw [-stealth, thick] (-7,0)--(-1,0);
	\draw [thick] (-6.05,3)--(-5.95, 3);
	\node [left] at (-6,5) {$\delta$};
	\node [left] at (-6,3) {$1$};
	\node [below] at (-1,0) {$1/p$};
	\node [below] at (-2,0) {$m$};
	\draw [thick] (-2,0.05)--(-2,-0.05);
	\draw [thick] (-6.05,-4)--(-5.95, -4);
	\node [left] at (-6,-4) {$\gamma-mn$};
	\draw [fill=aquamarine, fill opacity=0.5] (-6,-4)--(-2,-4)--(-4,3)--(-6,3)--cycle;
	\draw [dashed, thick] (-6,1)--(-3.4286,1);
	\node [left] at (-6,1) {$\tau$};
	\draw [fill=white] (-4,3) circle (0.08cm);
	\node [right] at (-3.5,2) {$\delta=\gamma-n/p$};
	\node[above] at (3,6) {$\gamma=1$};
	\draw [-stealth, thick] (1,-5)--(1,5);
	\draw [-stealth, thick] (0,0)--(6,0);
	\draw [thick] (0.95,3)--(1.05, 3);
	\node [left] at (1,5) {$\delta$};
	\node [left] at (1,3) {$1$};
	\node [below] at (6,0) {$1/p$};
	\node [below] at (5,0) {$m$};
	\draw [thick] (5,0.05)--(5,-0.05);
	\draw [thick] (0.95,-4)--(1.05, -4);
	\node [left] at (1,-4) {$\gamma-mn$};
	\draw [fill=aquamarine, fill opacity=0.5] (1,-4)--(5,-4)--(1,3)--cycle;
	\draw [dashed, thick] (1,1)--(2.152857,1);
	\node [left] at (1,1) {$\tau$};
	\draw [fill=white] (1,3) circle (0.08cm);
	\node [right] at (2,2) {$\delta=\gamma-n/p$};
	\node[above] at (10,6) {$\gamma<1$};
	\draw [-stealth, thick] (8,-5)--(8,5);
	\draw [-stealth, thick] (7,0)--(13,0);
	\draw [thick] (7.95,3)--(8.05, 3);
	\node [left] at (8,5) {$\delta$};
	\node [left] at (8,3) {$1$};
	\node [below] at (13,0) {$1/p$};
	\node [below] at (12,0) {$m$};
	\draw [thick] (12,0.05)--(12,-0.05);
	\draw [thick] (7.95,-4)--(8.05, -4);
	\node [left] at (8,-4) {$\gamma-mn$};
	\draw [fill=aquamarine, fill opacity=0.5] (8,-4)--(12,-4)--(8,2)--cycle;
	\draw [dashed, thick] (8,1)--(8.6667,1);
	\node [left] at (8,1) {$\tau$};
	\node [right] at (8.5,2) {$\delta=\gamma-n/p$};
	\end{tikzpicture}
\end{center}

\noindent In order to prove Theorem~\ref{teo: ejemplos para Hbb}, it will be useful the following lemma (see \cite{Pradolini01}).

\begin{lema}\label{lema: estimacion de la integral de |x|^a en una bola}
	If $R>0$, $B=B(x_B,R)$ is a ball in $\mathbb{R}^n$ and $\alpha>-n$ then
	\[\int_B |x|^{\alpha}\,dx\approx R^n\left(\max\{R,|x_B|\}\right)^\alpha.\]
\end{lema} 

\medskip

\begin{proof}[Proof of Theorem~\ref{teo: ejemplos para Hbb}]
Recalling that $\tau=(\gamma-mn)(1-1/m)+1/m$ is the number appearing in Lemma~\ref{lema: local implica global en Hbarra}, we shall split the proof into the following cases:
\begin{enumerate}[\rm(a)]
	\item \label{item: prueba de teo: ejemplos para Hbb - item a}$\gamma-mn<\delta<\tau\leq \gamma-n/p$;
	\item \label{item: prueba de teo: ejemplos para Hbb - item b}$\gamma-mn<\delta\leq \gamma-n/p<\tau$;
	\item \label{item: prueba de teo: ejemplos para Hbb - item c}$\gamma-mn<\delta=\tau<1<\gamma-n/p$; 
	\item \label{item: prueba de teo: ejemplos para Hbb - item d} $\gamma-mn<\delta=\tau<\gamma-n/p<1$;
	\item \label{item: prueba de teo: ejemplos para Hbb - item e} $\tau<\delta <\min\{1,\gamma-n/p\}$;
	\item \label{item: prueba de teo: ejemplos para Hbb - item f}$\delta=\gamma-mn$.
\end{enumerate}
Let us prove \eqref{item: prueba de teo: ejemplos para Hbb - item a}. Recall that $\mathcal{I}_1=\{i: p_i=1\}$, $\mathcal{I}_2=\{i: p_i>1\}$ and $m_j=\#\mathcal{I}_j$, for $j=1,2$. Since  $m_1<m$ by the restrictions on the parameters, we can take
\[0<\varepsilon<\frac{mn-\gamma+\delta}{m-m_1}.\]
For $1\leq i\leq m$ we define
\[\beta_i=\left\{\begin{array}{ccl}
0 & \textrm{ if } & i\in\mathcal{I}_1,\\
\frac{n}{p_i'}-\varepsilon & \textrm{ if } & i\in\mathcal{I}_2.
\end{array}
\right.
\]
 Let $\alpha=\sum_{i=1}^m \beta_i+\delta-\gamma+n/p>0$. Then we take
\[w(x)=|x|^{\alpha}\quad\textrm{ and }\quad v_{i}(x)=|x|^{\beta_i}.\]
By virtue of Lemma~\ref{lema: local implica global en Hbarra} it will be enough to show that $(w,\vec{v})$ verifies condition \eqref{eq: condicion local}. Let $B=B(x_B, R)$ and $|x_B|\leq R$, by Lemma~\ref{lema: estimacion de la integral de |x|^a en una bola} if $i\in\mathcal{I}_2$ we get
\begin{equation*}
\left(\frac{1}{|B|} \int_B v_i^{-p'_i}
\right)^{1/{p_i'}}=
\left(\frac{1}{|B|} \int_B |x|^{-\beta_ip'_i}\, dx
\right)^{1/{p_i'}}\approx R^{-\beta_i},
\end{equation*}
and $\left\|v_i^{-1}\mathcal{X}_B\right\|_\infty=1$ for $i\in\mathcal{I}_1$.
On the other hand, $\|w\mathcal{X}_B\|_{\infty}\lesssim R^\alpha$
since $\alpha>0$. Therefore,  
\[\frac{\|w\mathcal{X}_B\|_{\infty}}{|B|^{\delta/n-\gamma/n+1/p}}\prod_{i\in\mathcal{I}_1}\left\|v_i^{-1}\mathcal{X}_B\right\|_{\infty}\prod_{i\in \mathcal{I}_2}\left(\frac{1}{|B|}\int_B v_i^{-p_i'}\right)^{1/p_i'}\leq CR^{\alpha -\sum_{i=1 }^{m}\beta_i-\delta +\gamma -n/p}\leq C.\]

We now consider the case $|x_B|>R$. We have that 
\begin{equation*}
\|w\mathcal{X}_B\|_{\infty} \lesssim |x_B|^{\alpha}
\end{equation*}
whilst for $i\in\mathcal{I}_2$	
\begin{equation*}
\left(\frac{1}{|B|} \int_B |x|^{-\beta_i p'_i}\,dx \right)^{1/{p_i'}} \approx |x_B|^{-\beta_i}.
\end{equation*}
Consequently, since $\delta<\gamma-n/p$
\begin{equation*}
\frac{\|w\mathcal{X}_B\|_{\infty}}{|B|^{\delta/n-\gamma/n+1/p}}\prod_{i\in\mathcal{I}_1}\left\|v_i^{-1}\mathcal{X}_B\right\|_{\infty}\prod_{i\in \mathcal{I}_2}\left(\frac{1}{|B|}\int_B v_i^{-p_i'}\right)^{1/p_i'}\leq C
|x_B|^{\alpha - \sum_{i=1}^m \beta_i -\delta +\gamma -n/p} \leq C,
\end{equation*}
which completes the proof of \eqref{item: prueba de teo: ejemplos para Hbb - item a}. 

We now prove \eqref{item: prueba de teo: ejemplos para Hbb - item b}. In this case we take $w=1$ and $v_i=|x|^{\beta_i}$,  $\beta_i=(\gamma-\delta)/m-n/p_i$ for every $1\leq i\leq m$. By Lemma~\ref{lema: local implica global en Hbarra} it will be enough to prove that $(w,\vec{v})$ satisfies condition \eqref{eq: condicion local}. Pick a ball $B=B(x_B,R)$ and assume that $|x_B|\le R$. Observe that for every $i\in\mathcal{I}_1$ we get $\beta_i<0$, since we are assuming $\delta>\gamma-mn$. Then, for $i\in\mathcal{I}_1$ we get
\[\left\|v_i^{-1}\mathcal{X}_B\right\|_\infty\approx R^{-\beta_i}.\]
On the other hand, for $i\in\mathcal{I}_2$ we have $\beta_i<n/p_i'$, so Lemma~\ref{lema: estimacion de la integral de |x|^a en una bola} yields
\[\left(\frac{1}{|B|}\int_B v_i^{-p_i'}\right)^{1/p_i'}\approx R^{-\beta_i}.\]
These two estimates imply that
\[\frac{\|w\mathcal{X}_B\|_{\infty}}{|B|^{\delta/n-\gamma/n+1/p}}\prod_{i\in\mathcal{I}_1}\left\|v_i^{-1}\mathcal{X}_B\right\|_{\infty}\prod_{i\in \mathcal{I}_2}\left(\frac{1}{|B|}\int_B v_i^{-p_i'}\right)^{1/p_i'}\leq C \frac{R^{-\sum_{i=1}^m\beta_i}}{R^{\delta-\gamma+n/p}}=C.\]
If $|x_B|>R$, we have that $\left\|v_i^{-1}\mathcal{X}_B\right\|_{\infty}\lesssim |x_B|^{-\beta_i}$ and also 
\[\left(\frac{1}{|B|}\int_B v_i^{-p_i'}\right)^{1/p_i'}\approx |x_B|^{-\beta_i}\]
by Lemma~\ref{lema: estimacion de la integral de |x|^a en una bola}. Thus
\[\frac{\|w\mathcal{X}_B\|_{\infty}}{|B|^{\delta/n-\gamma/n+1/p}}\prod_{i\in\mathcal{I}_1}\left\|v_i^{-1}\mathcal{X}_B\right\|_{\infty}\prod_{i\in \mathcal{I}_2}\left(\frac{1}{|B|}\int_B v_i^{-p_i'}\right)^{1/p_i'}\leq C \frac{|x_B|^{-\sum_{i=1}^m\beta_i}}{R^{\delta-\gamma+n/p}}\leq C,\]
since $\delta\leq\gamma-n/p$. This concludes the proof of item~\eqref{item: prueba de teo: ejemplos para Hbb - item b}. 

In order to prove \eqref{item: prueba de teo: ejemplos para Hbb - item c} we pick $(\gamma-\tau)/m-n/p_i<\beta_i<n/p_i'$ for every $i\in\mathcal{I}_2$ and $\beta_i=0$ for $i\in\mathcal{I}_1$. Notice that this election is possible since $\gamma-\tau<mn$. We also take $\alpha=\sum_{i=1}^m\beta_i+\tau-\gamma+n/p$ and define $w(x)=|x|^\alpha$ and $v_i(x)=|x|^{\beta_i}$, for $1\leq i\leq m$. By virtue of Lemma~\ref{lema: equivalencia con local y global} it will be enough to show that condition \eqref{eq: condicion global} holds, since every $v_i$ is a doubling weight. We shall prove that
\[\frac{\|w\mathcal{X}_B\|_\infty}{|B|^{(\delta-1)/n}}\prod_{i\in\mathcal{I}_1}\left\|\frac{v_i^{-1}\mathcal{X}_{\mathbb{R}^n\backslash B}}{|x_B-\cdot|^{n-\gamma/m+1/m}}\right\|_\infty\prod_{i\in\mathcal{I}_2}\left(\int_{\mathbb{R}^n\backslash B} \frac{v_i^{-p_i'}(y)}{|x_B-y|^{(n-\gamma/m+1/m)p_i'}}\,dy\right)^{1/p_i'}\leq C\]
for every ball $B=B(x_B, R)$. We first notice that
\[\prod_{i\in\mathcal{I}_1}\left\|\frac{v_i^{-1}\mathcal{X}_{\mathbb{R}^n\backslash B}}{|x_B-\cdot|^{n-\gamma/m+1/m}}\right\|_\infty\leq R^{-\sum_{i\in\mathcal{I}_1}(n-\gamma/m+1/m)},\]
so there will be enough to show that
\begin{equation}\label{eq: teo: ejemplos para Hbb - item c - eq1}
R^{1-\delta-\sum_{i\in\mathcal{I}_1}(n-\gamma/m+1/m)}\|w\mathcal{X}_B\|_\infty\prod_{i\in\mathcal{I}_2}\left(\int_{\mathbb{R}^n\backslash B} \frac{v_i^{-p_i'}(y)}{|x_B-y|^{(n-\gamma/m+1/m)p_i'}}\,dy\right)^{1/p_i'}\leq C
\end{equation}
for every ball $B$. We shall first assume that $|x_B|\leq R$. For every $k\in\mathbb{N}$ we take $B_k=B\left(x_B, 2^kR\right)$ and for $i\in\mathcal{I}_2$, by Lemma~\ref{lema: estimacion de la integral de |x|^a en una bola}, we write
\begin{align*}
\left(\int_{\mathbb{R}^n\backslash B} \frac{v_i^{-p_i'}(y)}{|x_B-y|^{(n-\gamma/m+1/m)p_i'}}\,dy\right)^{1/p_i'}&\lesssim \sum_{k=1}^\infty (2^kR)^{-n+\gamma/m-1/m}\left(\int_{B_{k+1}\backslash B_k} |y|^{-\beta_ip_i'}\,dy\right)^{1/p_i'}\\
&\lesssim \sum_{k=1}^\infty (2^kR)^{-n+\gamma/m-1/m-\beta_i+n/p_i'}\\
&\lesssim R^{-n/p_i+\gamma/m-1/m-\beta_i},
\end{align*}\label{pag: estimacion para i fuera de I_1, |x_B|<=R}
since $-n/p_i+\gamma/m-1/m-\beta_i<0$ by the election of $\beta_i$. Then the left-hand side of \eqref{eq: teo: ejemplos para Hbb - item c - eq1} is bounded by
\[CR^{1-\delta-\sum_{i\in\mathcal{I}_1}(n-\gamma/m+1/m)+\alpha-\sum_{i\in\mathcal{I}_2}(n/p_i-\gamma/m+1/m+\beta_i)}=CR^{-\delta-n/p+\gamma+\alpha-\sum_{i=1}^m\beta_i}=C.\]

We now assume $|x_B|>R$. There exists a number $N$ such that $2^NR<|x_B|\leq 2^{N+1}R$. If $i\in\mathcal{I}_2$ we have that
\begin{align*}
\left(\int_{\mathbb{R}^n\backslash B} \frac{v_i^{-p_i'}(y)}{|x_B-y|^{(n-\gamma/m+1/m)p_i'}}\,dy\right)^{1/p_i'}&\lesssim \sum_{k=1}^\infty(2^{k}R)^{-n+\gamma/m-1/m}\left(\int_{B_k} |y|^{-\beta_ip_i'}\,dy\right)^{1/p_i'}\\
&=\sum_{k=1}^N+\sum_{k=N+1}^\infty\\
&=S_1^i+S_2^i.
\end{align*}\label{pag: estimacion del producto para i fuera de I_1, |x_B|>R}
Let $\theta_i=n/p_i+(1-\gamma)/m$, for $1\leq i\leq m$. We shall estimate the sum $S_1^i+S_2^i$, for $i\in\mathcal{I}_2$, by distinguishing into the cases $\theta_i<0$, $\theta_i=0$ and $\theta_i>0$. We shall first prove that if $\theta_i<0$, then
\begin{equation}\label{eq: teo: ejemplos para Hbb - item c - eq2}
S_j^i\leq C|x_B|^{-\beta_i-\theta_i},
\end{equation} 
for $j=1,2$. Indeed, by Lemma~\ref{lema: estimacion de la integral de |x|^a en una bola} we obtain
\begin{align*}
S_1^i&\lesssim \sum_{k=1}^N(2^{k}R)^{-n+\gamma/m-1/m+n/p_i'}|x_B|^{-\beta_i}\\
&\lesssim |x_B|^{-\beta_i}R^{-\theta_i}\sum_{k=1}^N 2^{-k\theta_i}\\
&\lesssim |x_B|^{-\beta_i}(2^NR)^{-\theta_i}\\
&\lesssim |x_B|^{-\beta_i-\theta_i},
\end{align*}
since $\theta_i<0$. For $S_2^i$ we apply again Lemma~\ref{lema: estimacion de la integral de |x|^a en una bola} in order to get
\begin{align*}
S_2^i&\lesssim \sum_{k=N+1}^\infty(2^{k}R)^{-n+\gamma/m-1/m+n/p_i'-\beta_i}\\
&\lesssim \sum_{k=N+1}^\infty \left(2^{k}R\right)^{-\beta_i-\theta_i}\\
&= \left(2^{N+1}R\right)^{-\beta_i-\theta_i}\sum_{k=0}^\infty 2^{-k(\beta_i+\theta_i)}\\
&\lesssim |x_B|^{-\beta_i-\theta_i},
\end{align*}
since $\theta_i+\beta_i=n/p_i+(1-\gamma)/m+\beta_i>0$.

Now assume that $\theta_i=0$. By proceeding similarly as in the previous case, we have
\[S_1^i\lesssim |x_B|^{-\beta_i}N\lesssim |x_B|^{-\beta_i}\log_2\left(\frac{|x_B|}{R}\right),\]
and
\[S_2^i\lesssim |x_B|^{-\beta_i}\]
since $\beta_i>0$ when $\theta_i=0$. Consequently,
\begin{equation}\label{eq: teo: ejemplos para Hbb - item c - eq3}
S_1^i+S_2^i\lesssim |x_B|^{-\beta_i}\left(1+\log_2\left(\frac{|x_B|}{R}\right)\right)\lesssim |x_B|^{-\beta_i}\log_2\left(\frac{|x_B|}{R}\right).
\end{equation}

We finally consider the case $\theta_i>0$. For $S_2^i$ we can proceed exactly as in the case $\theta_i<0$ and get the same bound. On the other hand, for $S_1^i$ we have that
\begin{align*}
S_1^i&\lesssim \sum_{k=1}^N(2^{k}R)^{-n+\gamma/m-1/m+n/p_i'}|x_B|^{-\beta_i}\\
&\lesssim |x_B|^{-\beta_i}R^{-\theta_i}\sum_{k=1}^N 2^{-k\theta_i}\\
&\lesssim |x_B|^{-\beta_i}\left(2^NR\right)^{-\theta_i}2^{N\theta_i}\,\frac{1-2^{-N\theta_i}}{1-2^{-\theta_i}}\\
&\lesssim |x_B|^{-\beta_i-\theta_i}2^{N\theta_i}.
\end{align*}\label{pag: estimacion de S_1^i y S_2^i,  theta_i>0}
Therefore, if $i\in\mathcal{I}_2$ and $\theta_i>0$ we get
\begin{equation}\label{eq: teo: ejemplos para Hbb - item c - eq4}
S_1^i+S_2^i\lesssim |x_B|^{-\beta_i-\theta_i}\left(1+2^{N\theta_i}\right)\lesssim 2^{N\theta_i}|x_B|^{-\beta_i-\theta_i}. 
\end{equation}
By combining \eqref{eq: teo: ejemplos para Hbb - item c - eq2}, \eqref{eq: teo: ejemplos para Hbb - item c - eq3} and \eqref{eq: teo: ejemplos para Hbb - item c - eq4} we obtain
\begin{align*}
\prod_{i\in\mathcal{I}_2}\left(\int_{\mathbb{R}^n\backslash B} \frac{v_i^{-p_i'}(y)}{|x_B-y|^{(n-\gamma/m+1/m)p_i'}}\,dy\right)^{1/p_i'}&\lesssim \prod_{i\in\mathcal{I}_2, \theta_i<0} |x_B|^{-\beta_i-\theta_i} \prod_{i\in\mathcal{I}_2, \theta_i=0} |x_B|^{-\beta_i}\log_2\left(\frac{|x_B|}{R}\right) \\
&\quad\times\prod_{i\in\mathcal{I}_2, \theta_i>0} |x_B|^{-\beta_i-\theta_i}2^{N\theta_i} \\
&\lesssim |x_B|^{-\sum_{i\in\mathcal{I}_2}(\beta_i+\theta_i)}2^{N\sum_{i\in\mathcal{I}_2,\theta_i> 0}\theta_i}\\
&\quad \times\left(\log_2\left(\frac{|x_B|}{R}\right)\right)^{\#\{i\in\mathcal{I}_2, \theta_i=0\}},
\end{align*}
so the left-hand side of \eqref{eq: teo: ejemplos para Hbb - item c - eq1} can be bounded by
\[CR^{1-\delta-(n-\gamma/m+1/m)m_1}|x_B|^{\alpha-\sum_{i\in\mathcal{I}_2}(\beta_i+\theta_i)}2^{N\sum_{i\in\mathcal{I}_2,\theta_i> 0}\theta_i}\left(\log_2\left(\frac{|x_B|}{R}\right)\right)^{\#\{i\in\mathcal{I}_2, \theta_i=0\}}\]
which is equal to
\begin{equation}\label{eq: teo: ejemplos para Hbb - item c - eq5}
C\left(\frac{|x_B|}{R}\right)^{\tau-1+(n-\gamma/m+1/m)m_1+\sum_{i\in\mathcal{I}_2,\theta_i> 0}\theta_i}\left(\log_2\left(\frac{|x_B|}{R}\right)\right)^{\#\{i\in\mathcal{I}_2, \theta_i=0\}}.
\end{equation}

Since $\theta_i<n+(1-\gamma)/m$ for $i\in\mathcal{I}_2$, there exists $\varepsilon>0$ that verifies
\[\sum_{i\in\mathcal{I}_2,\theta_i> 0}\theta_i+\varepsilon\#\{i\in\mathcal{I}_2, \theta_i=0\}\leq \left(n+\frac{1-\gamma}{m}\right)\#\{i\in\mathcal{I}_2, \theta_i>0\}.\]
By using the fact that $\log_2 t\lesssim \varepsilon^{-1}t^{\varepsilon}$ for every $t\geq 1$, we can majorize \eqref{eq: teo: ejemplos para Hbb - item c - eq5} by a constant factor provided that
\[\tau-1+\left(n+\frac{1-\gamma}{m}\right)\left(m_1+\#\{i\in\mathcal{I}_2: \theta_i>0\}\right)\leq \tau-1+\left(n+\frac{1-\gamma}{m}\right)(m-1)=0.\label{pag: estimacion del log_2(|x_B|/R)}\]
Indeed, if this last inequality did not hold, then we would have that $\theta_i>0$ for every $i\in\mathcal{I}_2$. We also observe that $\theta_i>0$ for $i\in\mathcal{I}_1$. This would lead to $n/p>\gamma-1$, a contradiction.

In order to prove \eqref{item: prueba de teo: ejemplos para Hbb - item d} we only need to consider two cases. If there exists some $i\in\mathcal{I}_2$ such that $\theta_i\leq 0$, the proof follows exactly as in \eqref{item: prueba de teo: ejemplos para Hbb - item c}. If not, that is $\theta_i>0$ for every  $i\in\mathcal{I}_2$, observe that
\begin{align*}
\tau-1+\left(n+\frac{1-\gamma}{m}\right)m_1+\sum_{i\in\mathcal{I}_2,\theta_i> 0}\theta_i&=\tau-1+\sum_{i=1}^m \left(\frac{n}{p_i}+\frac{1-\gamma}{m}\right)\\
&=\tau+\frac{n}{p}-\gamma\\
&<0,
\end{align*}
so we can choose $\varepsilon>0$ small enough so that the resulting exponent for $|x_B|/R$ in \eqref{eq: teo: ejemplos para Hbb - item c - eq5} is negative. This concludes the proof in this case.

We now proceed with the proof of \eqref{item: prueba de teo: ejemplos para Hbb - item e} we take
$\alpha=\delta-\tau>0$ and $\beta_i=(1-\tau)/m-\theta_i$, for every $i$. Then we define $w(x)=|x|^\alpha$ and $v_i=|x|^{\beta_i}$, $1\leq i\leq m$. These functions are locally integrable since $\alpha>0$ and $\beta_i<n/p_i'$. Furthermore, $\beta_i<0$ for $i\in\mathcal{I}_1$, so $v_i^{-1}\in \mathrm{RH}_\infty$ for these index. Then, by Lemma~\ref{lema: equivalencia con local y global}, it will be enough to show that condition \eqref{eq: condicion global} holds. Fix a ball $B=B(x_B, R)$ and assume that $|x_B|<R$. Then we get
\begin{equation}\label{eq: teo: ejemplos para Hbb - item e - eq1}
\frac{\|w\mathcal{X}_B\|_\infty}{|B|^{(\delta-1)/n}}\lesssim R^{1-\delta+\alpha}=R^{1-\tau}.
\end{equation}

On the other hand, if $i\in\mathcal{I}_1$ we have
\begin{align*}
\left\|\frac{v_i^{-1}\mathcal{X}_{\mathbb{R}^n\backslash B}}{|x_B-\cdot|^{n-\gamma/m+1/m}}\right\|_\infty&\lesssim \sum_{k=0}^\infty \left\|\frac{v_i^{-1}\mathcal{X}_{B_{k+1}\backslash B_k}}{|x_B-\cdot|^{n-\gamma/m+1/m}}\right\|_\infty\\
&\lesssim \sum_{k=0}^\infty \left(2^kR\right)^{-\beta_i-n+\gamma/m-1/m}\\
&\lesssim   R^{(\tau-1)/m},
\end{align*}
since $\tau<\delta<1$. This yields
\begin{equation}\label{eq: teo: ejemplos para Hbb - item e - eq2}
\prod_{i\in\mathcal{I}_1}\left\|\frac{v_i^{-1}\mathcal{X}_{\mathbb{R}^n\backslash B}}{|x_B-\cdot|^{n-\gamma/m+1/m}}\right\|_\infty\lesssim R^{m_1(\tau-1)/m}.
\end{equation}
Finally, since $\beta_i+\theta_i=(1-\tau)/m>0$ for $i\in\mathcal{I}_2$, we can proceed as in page~\pageref{pag: estimacion para i fuera de I_1, |x_B|<=R} to obtain
\begin{equation}\label{eq: teo: ejemplos para Hbb - item e - eq3}
\prod_{i\in\mathcal{I}_2}\left(\int_{\mathbb{R}^n\backslash B}\frac{v_i^{-p_i'}(y)}{|x_B-y|^{(n-\gamma/m+1/m)p_i'}}\,dy\right)^{1/p_i'}\lesssim R^{m_2(\tau-1)/m}.
\end{equation}
By combining \eqref{eq: teo: ejemplos para Hbb - item e - eq1}, \eqref{eq: teo: ejemplos para Hbb - item e - eq2} and \eqref{eq: teo: ejemplos para Hbb - item e - eq3}, the left-hand side of \eqref{eq: condicion global} is bounded by a constant $C$.
We now consider the case $|x_B|>R$. We have that
\begin{equation}\label{eq: teo: ejemplos para Hbb - item e - eq4}
\frac{\|w\mathcal{X}_B\|_\infty}{|B|^{(\delta-1)/n}}\lesssim R^{1-\delta}|x_B|^\alpha.
\end{equation}
Since $|x_B|>R$, there exists a number $N\in\mathbb{N}$ such that $2^NR<|x_B|\leq 2^{N+1}R$. For $i\in\mathcal{I}_1$ we write
\begin{align*}
\left\|\frac{v_i^{-1}\mathcal{X}_{\mathbb{R}^n\backslash B}}{|x_B-\cdot|^{n-\gamma/m+1/m}}\right\|_\infty&\lesssim \sum_{k=0}^N \left\|\frac{v_i^{-1}\mathcal{X}_{B_{k+1}\backslash B_k}}{|x_B-\cdot|^{n-\gamma/m+1/m}}\right\|_\infty+\sum_{k=N+1}^\infty \left\|\frac{v_i^{-1}\mathcal{X}_{B_{k+1}\backslash B_k}}{|x_B-\cdot|^{n-\gamma/m+1/m}}\right\|_\infty\\
&=S_1^i+S_2^i.
\end{align*}
By applying Lemma~\ref{lema: estimacion de la integral de |x|^a en una bola} and proceeding as in page~\pageref{pag: estimacion de S_1^i y S_2^i,  theta_i>0} with $p_i=1$ we have that
\begin{equation}\label{eq: teo: ejemplos para Hbb - item e - eq5}
\prod_{i\in\mathcal{I}_1}\left\|\frac{v_i^{-1}\mathcal{X}_{\mathbb{R}^n\backslash B}}{|x_B-\cdot|^{n-\gamma/m+1/m}}\right\|_\infty\lesssim |x_B|^{-\sum_{i\in\mathcal{I}_1}(\beta_i+\theta_i)}2^{N\sum_{i\in\mathcal{I}_1}\theta_i}.
\end{equation}
Finally, if $i\in\mathcal{I}_2$ our choice of $\beta_i$ allows us to follow the argument given in page~\pageref{pag: estimacion del producto para i fuera de I_1, |x_B|>R} to conclude that
\begin{align*}
\prod_{i\in\mathcal{I}_2}\left(\int_{\mathbb{R}^n\backslash B} \frac{v_i^{-p_i'}(y)}{|x_B-y|^{(n-\gamma/m+1/m)p_i'}}\,dy\right)^{1/p_i'}
&\lesssim |x_B|^{-\sum_{i\in\mathcal{I}_2}(\beta_i+\theta_i)}2^{N\sum_{i\in\mathcal{I}_2,\theta_i> 0}\theta_i}\\
&\quad \times\left(\log_2\left(\frac{|x_B|}{R}\right)\right)^{\#\{i\in\mathcal{I}_2, \theta_i=0\}}.
\end{align*}
By combining the inequality above with \eqref{eq: teo: ejemplos para Hbb - item e - eq4} and \eqref{eq: teo: ejemplos para Hbb - item e - eq5}, the left-hand side of \eqref{eq: condicion global} can be bounded by
\[CR^{1-\delta} |x_B|^{\alpha-\sum_{i=1}^m(\theta_i+\beta_i)}2^{N\sum_{i:\theta_i> 0}\theta_i}
\left(\log_2\left(\frac{|x_B|}{R}\right)\right)^{\#\{i\in\mathcal{I}_2, \theta_i=0\}}\]
or equivalently by
\[C\left(\frac{R}{|x_B|}\right)^{1-\delta-\sum_{i: \theta_i>0}\theta_i}\left(\log_2\left(\frac{|x_B|}{R}\right)\right)^{\#\{i\in\mathcal{I}_2, \theta_i=0\}}.\]
If $\theta_i<0$ for every $i$ then the exponent of $R/|x_B|$ is positive. On the other hand, if $\theta_i\geq 0$ for every $i$ then
\[1-\delta-\sum_{i: \theta_i>0}\theta_i=1-\delta-\sum_{i=1}^m \theta_i=1-\delta-\frac{n}{p}+\gamma-1>0,\]
since $\delta<\gamma-n/p$. In both cases we can repeat the argument given in page~\pageref{pag: estimacion del log_2(|x_B|/R)} to conclude that $(w,\vec{v})$ belongs to $\mathbb{H}_m(\vec{p},\gamma,\delta)$. Let us observe that, for example, if $\gamma\leq1$ then every $\theta_i$ is nonnegative.

We finish with the proof of item~\eqref{item: prueba de teo: ejemplos para Hbb - item f}. In this case we fix $\alpha>0$ and take $w(x)=\left(1+|x|^\alpha\right)^{-m_1}$. If $g_i$ are nonnegative fixed functions in $L^{p_i'}(\mathbb{R}^n)$ for $i\in\mathcal{I}_2$, we define
\[v_i(x)=\left\{
\begin{array}{ccl}
e^{|x|}&\textrm{ if }& i\in\mathcal{I}_1,\\
g_i^{-1}&\textrm{ if }& i\in\mathcal{I}_2.
\end{array}
\right.\] 
Fix a ball $B=B(x_B, R)$. It is enough to check condition \eqref{eq: condicion local}, since $\delta=\gamma-mn<\tau$. Notice that 
\[\left\|w\mathcal{X}_B\right\|_\infty\prod_{i\in\mathcal{I}_1}\left\|v_i^{-1}\mathcal{X}_B\right\|_\infty\leq \prod_{i\in\mathcal{I}_1}\left\|(1+|\cdot|^{\alpha})^{-1}\mathcal{X}_B\right\|_\infty\left\|e^{-|\cdot|}\mathcal{X}_B\right\|_\infty\leq 1.\]
Therefore,
\begin{align*}
\frac{\|w\mathcal{X}_B\|_{\infty}}{|B|^{\delta/n-\gamma/n+1/p}}\prod_{i\in\mathcal{I}_1}\left\|v_i^{-1}\mathcal{X}_B\right\|_{\infty}\prod_{i\in \mathcal{I}_2}\left(\frac{1}{|B|}\int_B v_i^{-p_i'}\right)^{1/p_i'}&\leq |B|^{-\delta/n+\gamma/n-m}\prod_{i\in\mathcal{I}_2}\|g_i\|_{p_i'}\\
&\leq C\prod_{i\in\mathcal{I}_2}\|g_i\|_{p_i'},
\end{align*}
for every ball $B$. This concludes the proof of \eqref{item: prueba de teo: ejemplos para Hbb - item f}.
\end{proof}

We finish with the proof of Theorem~\ref{teo: caso pesos iguales}.

\begin{proof}
	Let $\vec{v}\in \mathbb{H}_m({\vec{p},\gamma,\delta})$ and $B$ be a ball. Since condition \eqref{eq: condicion local} holds, then we have that
	\begin{equation*}
		|B|^{\gamma/n-\delta/n-1/p}\left\|\mathcal{X}_B\prod_{i=1}^m v_i\right\|_\infty\prod_{i\in\mathcal{I}_1}\left\|v_i^{-1}\mathcal X_B\right\|_{\infty}\, \prod_{i\in\mathcal{I}_2}\left(\frac{1}{|B|}\int_{B} v_i^{-p_i'}(y)\,dy\right)^{1/p_i'}\leq C,	
	\end{equation*}
	for some positive constant $C$.	This implies that 
	\begin{equation*}
		{|B|^{\gamma/n-\delta/n-1/p}}\prod_{i\in\mathcal{I}_1}\left\|v_i^{-1}\mathcal X_B\right\|_{\infty}\, \prod_{i\in\mathcal{I}_2}\left(\frac{1}{|B|}\int_{B} v_i^{-p_i'}(y)\,dy\right)^{1/p_i'}\leq C \inf_{B}\prod_{i=1}^m v_i^{-1}.
	\end{equation*}
	Thus we deduce that
	\begin{equation*}
		{|B|^{\gamma/n-\delta/n-1/p}}\prod_{i=1}^m\inf_{B}v_i^{-1} \leq C \inf_{B}\prod_{i=1}^m v_i^{-1}.
	\end{equation*}
	Since $\prod_{i=1}^m \inf_{B}v_i^{-1}\leq \inf_{B}\prod_{i=1}^m v_i^{-1}$, we get that 
	\begin{equation*}
		|B|^{\gamma/n-\delta/n-1/p}\leq C
	\end{equation*}
	holds for every ball $B$, so that we must have $\delta=\gamma-n/p$. 
\end{proof}

		
\def\cprime{$'$}
\providecommand{\bysame}{\leavevmode\hbox to3em{\hrulefill}\thinspace}
\providecommand{\MR}{\relax\ifhmode\unskip\space\fi MR }
\providecommand{\MRhref}[2]{%
	\href{http://www.ams.org/mathscinet-getitem?mr=#1}{#2}
}
\providecommand{\href}[2]{#2}


\begin{thebibliography}{1}
	
	\bibitem{AHIV}
	H.~Aimar, S.~Hartzstein, B.~Iaffei, and B.~Viviani, \emph{The {R}iesz potential
		as a multilinear operator into general {$\rm BMO_\beta$} spaces}, vol. 173,
	2011, Problems in mathematical analysis. No. 55, pp.~643--655. \MR{2839849}
	
	\bibitem{CPR16}
	Adri\'{a}n Cabral, Gladis Pradolini, and Wilfredo Ramos, \emph{Extrapolation
		and weighted norm inequalities between {L}ebesgue and {L}ipschitz spaces in
		the variable exponent context}, J. Math. Anal. Appl. \textbf{436} (2016),
	no.~1, 620--636.
	
	\bibitem{HSV}
	E.~Harboure, O.~Salinas, and B.~Viviani, \emph{Boundedness of the fractional
		integral on weighted {L}ebesgue and {L}ipschitz spaces}, Trans. Amer. Math.
	Soc. \textbf{349} (1997), no.~1, 235--255.
	
	\bibitem{Moen09}
	Kabe Moen, \emph{Weighted inequalities for multilinear fractional integral
		operators}, Collect. Math. \textbf{60} (2009), no.~2, 213--238.
	
	\bibitem{Muckenhoupt-Wheeden74}
	B.~Muckenhoupt and R.~Wheeden, \emph{Weighted norm inequalities for fractional
		integrals}, Trans. Amer. Math. Soc. \textbf{192} (1974), 261--274.
	
	\bibitem{Prado01cal}
	Gladis Pradolini, \emph{A class of pairs of weights related to the boundedness
		of the fractional integral operator between {$L^p$} and {L}ipschitz spaces},
	Comment. Math. Univ. Carolin. \textbf{42} (2001), no.~1, 133--152.
	\MR{1825378}
	
	\bibitem{Pradolini01}
	\bysame, \emph{Two-weighted norm inequalities for the fractional integral
		operator between {$L^p$} and {L}ipschitz spaces}, Comment. Math. (Prace Mat.)
	\textbf{41} (2001), 147--169.
	
	\bibitem{Pradolini10}
	\bysame, \emph{Weighted inequalities and pointwise estimates for the
		multilinear fractional integral and maximal operators}, J. Math. Anal. Appl.
	\textbf{367} (2010), no.~2, 640--656.
	
	\bibitem{PRR21}
	Gladis Pradolini, Wilfredo Ramos, and Jorgelina Recchi, \emph{On the optimal
		numerical parameters related with two weighted estimates for commutators of
		classical operators and extrapolation results}, Collect. Math. \textbf{72}
	(2021), no.~1, 229--259.
	
\end{thebibliography}
\end{document}